\documentclass[reqno]{amsart}
\usepackage{amssymb}
\usepackage[dvips]{epsfig}
\usepackage{graphicx}
\usepackage{color}
\usepackage{amsmath}
\usepackage{amssymb}
\usepackage{amsfonts}
\usepackage{amsthm}
\usepackage{mathrsfs}
\usepackage{tikz-cd}
\theoremstyle{definition}
\newtheorem{Def}{Definition}[section]

\theoremstyle{plain}
\newtheorem{Thm}[Def]{Theorem}
\newtheorem{Conj}[Def]{Conjecture}

\newtheorem{Lem}[Def]{Lemma}

\newtheorem{Cor}[Def]{Corollary}
\newtheorem*{Claim}{Claim}
\theoremstyle{remark}
\newtheorem{Rem}[Def]{Remark}

\usepackage{hyperref}
\hypersetup{colorlinks=true,pdfstartview=FitV,linkcolor=magenta,citecolor=cyan}
\usepackage{bm}

\begin{document}
	\title{The Euler classes of taut foliations on graph manifolds}
	\author{Yaoping Xie}
	\address{Beijing International Center for Mathematical Research,
		No. 5 Yiheyuan Road, Haidian District,
		Beijing 100871,
		China}
	\email{2301110020@pku.edu.cn}
	\date{\today}
	
	\begin{abstract}
		In this paper, we determine the set of Euler classes of all taut foliations for a family of closed graph manifolds. As a consequence, for any closed aspherical graph manifold $Y$, every integral class in the closed dual Thurston unit ball can be virtually realized as the Euler class of some taut foliation. In particular, $Y$ virtually admits a taut foliation with vanishing Euler class. We present examples of graph manifolds with arbitrarily large first Betti number, which admit no taut foliations with vanishing real Euler classes. This shows that the virtual requirement is necessary. We also provide examples of graph manifolds with arbitrarily large first Betti number, for which every taut foliation has vanishing real Euler class, but none has vanishing integral Euler class. This illustrates the subtle difference between real Euler class zero and integral Euler class zero.
	\end{abstract}
	
	\maketitle
	
	\section{Introduction}
	
	For any closed orientable irreducible $3$-manifold $Y$, the dual Thurston norm unit ball $\mathcal{B}_{\mathrm{Th}^{*}}(Y)$ is a compact polytope in $H^{2}(Y;\mathbb{R})$, which is the convex hull of finitely many vertices in the even lattice $2H^{2}(Y;\mathbb{Z})/\mathrm{tor}$. Thurston proves that the real Euler class of any taut foliation on $Y$ must be an even lattice point in $\mathcal{B}_{\mathrm{Th}^{*}}(Y)$ \cite{Th86}. It is natural to ask conversely that which even lattice points in $\mathcal{B}_{\mathrm{Th}^{*}}(Y)$ can be realized as the real Euler classes of taut foliations on $Y$.
	
	Thurston conjectures that if $Y$ is atoroidal, then every even lattice point in $\partial\mathcal{B}_{\mathrm{Th}}^{*}(Y)$ is the real Euler class of some taut foliation on $Y$ \cite[Conjecture 3]{Th86}. This is known as the \textit{Euler class one conjecture}. Using the sutured manifold hierarchy, Gabai proves the partial result that each vertex of $\partial\mathcal{B}_{\mathrm{Th}}^{*}(Y)$ is the real Euler class of some taut foliation on $Y$ (see, for example, \cite[Theorem 1.4]{GabYaz20}). In general, the Euler class one conjecture is false. Yazdi constructs explicit counter-examples among closed orientable hyperbolic $3$-manifolds of first Betti number $2$ \cite{Yazdi}. Later, Liu proves that every closed orientable hyperbolic $3$-manifold has a finite cover for which the Euler class one conjecture is false \cite{Liu2409}.
	
	After disproving the Euler class one conjecture, Yazdi poses a virtual version called the \textit{virtual Euler class one conjecture} \cite[Question 9.4]{Yazdi}.
	
	\begin{Conj}[Yazdi]
		Let $Y$ be a closed orientable hyperbolic $3$-manifold. For any rational point  $w\in\partial\mathcal{B}_{\mathrm{Th}^{*}}(Y)$, there exists a finite cover $\pi\colon\tilde{Y}\rightarrow Y$ such that $\pi^{*}w$ is the real Euler class of some taut foliation on $\tilde{Y}$. 
	\end{Conj}
	
	In \cite{Liu2411}, Liu gives a criterion based on the Alexander polynomial for a rational point $w\in\partial\mathcal{B}_{\mathrm{Th}^{*}}(Y)$ to be virtually realized as the Euler class of a taut foliation. Using the criterion, Liu constructs examples with first Betti number $2$ or $3$, supporting the virtual class-one conjecture.
	
	While the above-mentioned results all focus on classes in the boundary of $\mathcal{B}_{\mathrm{Th}^{*}}(Y)$, it is natural to broaden the perspective to classes in the interior of $\mathcal{B}_{\mathrm{Th}^{*}}(Y)$. We pose the following more general conjecture:
	
	\begin{Conj}\label{general_virtual_Euler_class_one}
		Let $Y$ be a closed orientable irreducible $3$-manifold with $b_{1}(Y)\geq 1$. For any rational class  $w\in\mathcal{B}_{\mathrm{Th}^{*}}(Y)$, there exists a finite cover $\pi\colon\tilde{Y}\rightarrow Y$ such that $\pi^{*}w$ is the real Euler class of some taut foliation on $\tilde{Y}$.
	\end{Conj}
	
	For the (virtual) realization problem, interior classes exhibit a quite different flavor from boundary classes. Indeed, if a taut foliation has any compact leaf of negative Euler characteristic, then the Euler class must lie in the boundary of $\mathcal{B}_{\mathrm{Th}^{*}}(Y)$. Thus, to construct a taut foliation whose Euler class lies in the interior of $\mathcal{B}_{\mathrm{Th}^{*}}(Y)$, one cannot follow the idea of Gabai's sutured manifold hierarchy to pick a taut surface of negative characteristic as a leaf.
	
	In this paper, we focus on the graph manifold case, which serves as the first step in the study of the (virtual) realization problem. 
    Our main result is as follows.
	
	\begin{Thm}\label{Main_Thm}
		Conjecture \ref{general_virtual_Euler_class_one} holds for every closed aspherical graph manifold.
	\end{Thm}
	
	Here, asphericity is just a mild condition, since a closed graph manifold is not aspherical if and only if it supports either $\mathbb{S}^{3}$-geometry or $\mathbb{S}^{2}\times\mathbb{R}$-geometry. Any manifold supporting $\mathbb{S}^{3}$-geometry is a L-space, while any manifold supporting $\mathbb{S}^{2}\times\mathbb{R}$-geometry is reducible. Both of them do not admit taut foliations, so we need to exclude them.
	
	\begin{Cor}\label{integral_virtual_realization}
		Let $Y$ be a closed aspherical graph manifold. Then for every $w\in H^{2}(Y,\mathbb{Z})$ with $\|w\|_{\mathrm{Th}^{*}}\leq 1$, there exists a finite cover $\pi\colon\tilde{Y}\rightarrow Y$ and a taut foliation $\mathcal{F}$ on $\tilde{Y}$ such that $e(\mathcal{F})=\pi^{*}w$.
	\end{Cor}
	
	\begin{proof}
		By Theorem \ref{Main_Thm}, there exists a finite cover $\pi'\colon Y'\rightarrow Y$ such that $w'=(\pi')^{*}w$ is the real Euler class of a taut foliation $\mathcal{F}'$ on $Y'$. Then $w'$ differs from $e(\mathcal{F}')$ by a torsion class in $H^{2}(Y,\mathbb{Z})$. By \cite[Lemma 2.9]{LeeLip08}, there is a further finite cover $\pi''\colon\tilde{Y}\rightarrow Y'$ such that $\pi^{*}(w'-e(\mathcal{F}'))=0$. Then $(\pi'')^{*}w'=(\pi'')^{*}(e(\mathcal{F}'))=e((\pi'')^{*}\mathcal{F})$. So $\pi=\pi'\circ\pi''\colon\tilde{Y}\rightarrow Y$ is a desired cover.
	\end{proof}
	
	In particular, setting $w=0$, we obtain the following corollary:
	
	\begin{Cor}
		Let $Y$ be a closed aspherical graph manifold. Then there exists a finite cover $\tilde{Y}\rightarrow Y$ such that $\tilde{Y}$ admits a taut foliation with vanishing Euler class.
	\end{Cor}
	
	In the above corollary, the virtual requirement is necessary. Otherwise, the conclusion fails, as shown in the following theorem.
	
	\begin{Thm}\label{counterexample_real_Euler_class_zero}
		There exist closed graph manifolds with arbitrarily large first Betti number, which admit no taut foliations with vanishing real Euler classes.
	\end{Thm}
	
	There are subtle differences between zero Euler class and zero real Euler class as well. Recently, Boyer, Gordon, Hu and McCoy construct infinitely many (small Seifert-fibered, hyperbolic and toroidal) rational homology $3$-spheres that admit taut foliations (which must have vanishing real Euler classes), but none with vanishing integral Euler class \cite{BGHM}. The following theorem provides examples of graph manifolds with positive first betti numbers which have similar properties.
	
	\begin{Thm}\label{counterexample_integral_Euler_class_zero}
		There exist closed graph manifolds with arbitrarily large first Betti number, for which every taut foliation has vanishing real Euler class, but none has vanishing integral Euler class.
	\end{Thm}
	
	\subsection*{Outline of the proof}
	Let $Y$ be a closed graph manifold such that each JSJ piece is a product of a circle and a positive genus surface.
	The crux is to determine the set of Euler classes of all taut foliations on such manifolds. This would be done in Theorem \ref{graph_taut_foliation_distribution}.
	Then Theorem \ref{Main_Thm}, Theorem \ref{counterexample_real_Euler_class_zero} and Theorem \ref{counterexample_integral_Euler_class_zero} can be derived from Theorem \ref{graph_taut_foliation_distribution}, together with some knowledge of Thurston norms and virtual reductions of graph manifolds.
	
	The idea of constructing taut foliations on $Y$ is to pick an oriented essential lamination $\Lambda$ on $Y$ and foliate each complementary region to produce a foliation on the whole $Y$.
	Concretely, let $N_{i}=\Sigma_{i}\times\mathbb{S}^{1}\ (1\leq i\leq n)$ be the JSJ pieces of $Y$. Pick an oriented $1$-dimensional essential lamination $\lambda_{i}$ on each $\Sigma_{i}$ such that the complementary region of $\lambda_{i}$ consists of $4$-gons and holed $2$-gons. Then we obtain an oriented $2$-dimensional essential lamination $\Lambda=\sqcup_{i=1}^{n}\lambda_{i}\times\mathbb{S}^{1}$ on $Y$. Let $C$ be the complementary region of $\Lambda$. The essential part of $C$ is a union of taut sutured solid tori and taut torus shells. Each component of $C$ can be tautly foliated by stacks of saddles in various ways, depending on how the leaves wind around the tangential boundaries. The taut foliation on $C$ together with $\Lambda$ form a taut foliation on $Y$. Some calculations show that any class of the form $\sum_{i=1}^{n}a_{i}\cdot PD[l_{i}]$ with $|a_{i}|\leq-\chi_{i},\ a_{i}\equiv\chi_{i}\mod 2$ can be realized as the Euler class of some taut foliation obtained in this way, where $\chi_{i}=\chi(\Sigma_{i})$ and $l_{i}$ is an oriented $\mathbb{S}^{1}$-fiber of $N_{i}$. Meanwhile, results from \cite{BNR97} and \cite{BR99} show that the Euler class of any taut foliation on $Y$ should have this form. Therefore, the set of Euler classes of all taut foliations on $Y$ exactly consists of classes of such forms.
	
	\subsection*{Organization}
	In Section \ref{Pre}, we present definitions and basic properties of the related objects: Thurston norm, graph manifolds, relative Euler class, foliations, and sutured manifolds. In Section \ref{Standard_lamination_on_surface}, we assign essential laminations to all compact surfaces, which would be used to construct essential laminations on graph manifolds. In Section \ref{SST}, we construct taut foliations on sutured solid tori and sutured torus shells, and calculate their relative Euler classes. In Section \ref{non-geometric_case}, we prove the main theorems.
	
	\subsection*{Convention}
	In this paper, $\mathbb{T}$ stands for an oriented torus. $\mathbb{S}^{1}$ stands for an oriented circle.
	Denote $\mathbb{H}^{n}=\{(x_{1},\ldots,x_{n})\in\mathbb{R}^{n}:x_{n}\geq0\}$.
	All vector bundles $E\rightarrow X$ are supposed to be equipped with metrics.
	All 2-dimensional foliations $\mathcal{F}$ on oriented $3$-manifolds are supposed to be oriented and $C^{\infty,0}$ (see subsection \ref{foliation} for the definition of $C^{\infty,0}$).
	
	\subsection*{Acknowledgement}
	The author thanks his advisor Yi Liu for proposing the problem and providing suggestions throughout the research, and thanks Qingfeng Lyu, Jianru Duan and Xingpie Liu for valuable conversations.
	
	\section{Preliminaries}\label{Pre}
	
	\subsection{The Thurston norm}
	
	For any compact orientable $3$-manifold $Y$ with toroidal boundary, the Thurston norm is a seminorm on $H_{2}(Y,\partial Y;\mathbb{R})$ introduced by Thurston \cite{Th86}. For any connected compact oriented surface $S$, the \textit{complexity of $S$}, denoted by $x(S)$, is defined to be the nonnegative integer $\max\{-\chi(S),0\}$. The complexity of a compact oriented surface refers to the sum of the complexity of all connected components of $S$.
	For $\alpha\in H_{2}(Y,\partial Y;\mathbb{Z})$, assign $\|\alpha\|_{\mathrm{Th}}$ to be the infimum of $x(S)$ among all embedded compact oriented surfaces $S$ representing $\alpha$. Thurston shows that $\|\cdot\|_{\mathrm{Th}}$ extends uniquely to a seminorm, still denoted as $\|\cdot\|_{\mathrm{Th}}$, on $H_{2}(Y,\partial Y;\mathbb{R})$. Let $\|\cdot\|_{\mathrm{Th}^{*}}$ be the dual (partial) norm on $H^{2}(Y,\partial Y;\mathbb{R})$ defined by
	$ \|w\|_{\mathrm{Th}^{*}}=\sup\{w(\alpha)\colon\|\alpha\|_{\mathrm{Th}}=1\} $. $\|\cdot\|_{\mathrm{Th}}$ and $\|\cdot\|_{\mathrm{Th}^{*}}$ are called the \textit{Thurston norm} and the \textit{dual Thurston norm} of $Y$ respectively.
	Denote the closed unit balls of $\|\cdot\|_{\mathrm{Th}}$ and $\|\cdot\|_{\mathrm{Th}^{*}}$ by $\mathcal{B}_{\mathrm{Th}}(Y)$ and $\mathcal{B}_{\mathrm{Th}^{*}}(Y)$, respectively.
	
	\subsection{Graph manifolds}
	
	\subsubsection{JSJ decomposition}
	
	Let $Y$ be a closed irreducible $3$-manifold. There exists a (possibly empty) minimal union $\mathcal{T}$ of disjoint essential tori, unique up to isotopy, such that each component of $Y\backslash\mathcal{T}$ is either Seifert-fibered, a torus (semi-)bundle or atoroidal. The canonical decomposition of $Y$ given by $\mathcal{T}$ is called the \textit{JSJ decomposition} \cite{JS79}, \cite{Joh79}. $Y$ is called \textit{geometric} if and only if $\mathcal{T}$ is empty. We refer to the components of $\mathcal{T}$ as the \textit{JSJ tori} of $Y$, and refer to the components of $Y\backslash\mathcal{T}$ as the \textit{JSJ pieces} of $Y$.
	
	The geometrization of $3$-manifolds implies that each JSJ piece supports exactly one of the following eight geometries: $\mathbb{S}^{3}$, $\mathbb{S}^{2}\times\mathbb{R}$, $\mathbb{R}^{3}$, $\mathrm{Nil}$, $\mathbb{H}^{2}\times\mathbb{R}$, $\widetilde{\mathrm{SL}_{2}}$, $\mathrm{Sol}$, $\mathbb{H}^{3}$. The first six geometries correspond to Seifert-fibered manifolds, the seventh geometry corresponds to torus (semi-)bundles of Anosov type and, the last geometry is known as the hyperbolic geometry. See \cite[\S 11.5, \S 12]{Mar} for details on JSJ decomposition and geometrization.
	
	We call $Y$ a \textit{graph manifold}, if the JSJ decomposition of $Y$ contains no hyperbolic pieces. In other words, either $Y$ supports $\mathrm{Sol}$-geometry or each JSJ piece of $Y$ is Seifert-fibered. Be aware that we include Seifert-fibered manifolds and torus (semi)-bundles into the definition of graph manifolds, while some literature excludes them.
	
	We need the following virtual reduction lemma.
	
	\begin{Lem}[{\cite[Lemma 2.1]{Liu}}]\label{first_virtual_reduction}
		Let $Y$ be a graph manifold which is not geometric. Then there exists a finite cover $\tilde{Y}\rightarrow Y$ such that each JSJ piece of $\tilde{Y}$ is a product of $\mathbb{S}^{1}$ and a compact surface of positive genus.
	\end{Lem}
	
	\subsubsection{JSJ characteristic covers}
	Let $Y$ be a closed orientable irreducible $3$-manifold. For a positive integer $m$, we say a finite cover $\tilde{Y}\rightarrow Y$ is \textit{JSJ $m$-characteristic}, if for any JSJ torus $T$ of $Y$, every elevation $\tilde{T}$ of $T$ is a $m$-characteristic cover of $T$, which means that every slope of $\tilde{T}$ covers a slope of $T$ with degree $m$.
	We require the following result about existence of JSJ characteristic covers.
	
	\begin{Lem}[{\cite[Proposition 4.2]{Liu}}]\label{characteristic_cover}
		Let $Y$ be a closed orientable irreducible $3$-manifold. Then there is a positive integer $m_{0}$ satisfying that for any positive integral multiple $m$ of $m_{0}$, there is a JSJ $m$-characteristic cover of $Y$. 
	\end{Lem}
	
	\subsubsection{Thurston norms of graph manifolds}
	
	Let $Y$ be a closed graph manifold. Let $N_{1},\dots,N_{n}$ be the JSJ pieces of $Y$. For each $i$, let $\chi_{i}=\chi^{orb}(N_{i})$ and $l_{i}$ be an oriented regular fiber of $N_{i}$. The Thurston norm of $Y$ can be explicitly described as follows.
	
	\begin{Lem}[{\cite[Lemma 3.1]{Cigna}}]\label{Thuston_norm_graph_manifold}
		Suppose that $\chi_{i}<0$ for each $i$. For every $\alpha\in H_{2}(Y;\mathbb{R})$, the following formula holds:
		\[ \|\alpha\|_{\mathrm{Th}}=-\sum\limits_{i=1}^{n}\chi_{i}| PD[l_{i}](\alpha)|. \]
	\end{Lem}
	
	The following Lemma characterizes the rational points in the dual Thurston norm unit ball of $Y$.
	
	\begin{Lem}\label{Char_rational_point}
		$\mathcal{B}_{\mathrm{Th}^{*}}(Y)$ is the convex hull of $A=\left\{\sum\limits_{i=1}^{n}\pm\chi_{i}PD[l_{i}]\right\}$. Every rational point $w$ in $\mathcal{B}_{\mathrm{Th}^{*}}(Y)$ can be expressed as $w=\sum\limits_{i=1}^{n}a_{i}PD[l_{i}]$ where $a_{i}\in\mathbb{Q},\ |a_{i}|\leq-\chi_{i}$.
	\end{Lem}
	
	\begin{proof}
		Denote by $\mathcal{C}(A)$ the convex hull of $A$. By Lemma \ref{Thuston_norm_graph_manifold}, we have
        \[
            \|\alpha\|_{\mathrm{Th}}=\max\limits_{w\in \mathcal{C}(A)}\{w(\alpha)\},\ \forall\alpha\in H_{2}(Y;\mathbb{R}).
        \]
        For any $w\in\mathcal{C}(A)$, we have $|w(\alpha)|\leq\|\alpha\|_{\mathrm{Th}},\ \forall\alpha\in H_{2}(Y;\mathbb{R})$. So $\|w\|_{\mathrm{Th}^{*}}\leq 1$. We obtain $\mathcal{C}(A)\subset\mathcal{B}_{\mathrm{Th}^{*}}(Y)$. Suppose that there is some point $w_{0}\in\mathcal{B}_{\mathrm{Th}^{*}}(Y)\backslash\mathcal{C}(A)$. By the hyperplane separation theorem, there is a $\alpha\in H_{2}(Y;\mathbb{R})\cong H^{2}(Y;\mathbb{R})^{*}$ such that $\alpha(w_{0})>\max\limits_{w\in \mathcal{C}(A)}\{w(\alpha)\}$. Then we have
        \[
            \|\alpha\|_{\mathrm{Th}}=\max\limits_{w\in \mathcal{C}(A)}\{w(\alpha)\}<w_{0}(\alpha)\leq \|w_{0}\|_{\mathrm{Th}^{*}}\|\alpha\|_{\mathrm{Th}}\leq\|\alpha\|_{\mathrm{Th}},
        \]
        which leads to a contradiction. So $\mathcal{B}_{\mathrm{Th}^{*}}(Y)=\mathcal{C}(A)$. In particular, every vertex of $\mathcal{B}_{\mathrm{Th}^{*}}(Y)$ is in $A$. Every rational point in $\mathcal{B}_{\mathrm{Th}^{*}}(Y)$ is a rational convex combination of vertices of $\mathcal{B}_{\mathrm{Th}^{*}}(Y)$, so it can be expressed as $\sum\limits_{i=1}^{n}a_{i}PD[l_{i}]$, where $a_{i}\in\mathbb{Q},\ |a_{i}|\leq-\chi_{i}$.
	\end{proof}
	
	\subsection{The relative Euler class}
	Let $(X,Y)$ be a CW pair. Then there is a natural isomorphism $H^{*}(X/Y)\cong H^{*}(X,Y)$. Suppose that $E$ is an oriented circle bundle over $X$ with $E|_{Y}$ trivial. Let $s$ be a section of $E|_{Y}$. Then $s$ gives a trivialization $\Phi_{s}\colon E|_{Y}\rightarrow Y\times\mathbb{S}^{1}$. Let $p\colon Y\times\mathbb{S}^{1}\rightarrow\mathbb{S}^{1}$ be the projection. Consider the quotient circle bundle $E_{s}$ over $X/Y$ given by $E_{s}=E/\sim$, where $v_{1}\sim v_{2}$ if and only if $v_{1},v_{2}\in E|_{Y}$ and $p(\Phi(v_{1}))=p(\Phi(v_{2}))$. Then the \textit{relative Euler class} $e(E,s)\in H^{2}(X,Y)$ is defined as the class $e(E_{s})\in H^{2}(X/Y)\cong H^{2}(X,Y)$.
    Suppose that $V$ is a oriented $2$-plane bundle over $X$ with $V|_{Y}$ trivial. Let $s$ be a unit section of $V|_{Y}$ ($s_{y}$ has unit length for each $y\in Y$). Then we define the \textit{relative Euler class} $e(V,s)$ as $e(E,s)$, where $E$ is the associated unit circle bundle of $V$.
	We have the following \textit{gluing formula}.
	
	\begin{Lem}[Gluing formula]\label{Gluing_formula}
		Let $(X,Y)$ be a CW pair.
		Let $X_{1},\ldots,X_{n}$ be subcomplexes of $X$ such that $X=X_{1}\cup\cdots\cup X_{n}$ and $X_{i}\cap X_{j}\subset Y,\ \forall 1\leq i<j\leq n$.
		In other words, $X/Y=\bigvee_{i=1}^{n}X_{i}/(X_{i}\cap Y)$.
		So we have a natural isomorphism
		\[H^{2}(X,Y)\cong H^{2}(X/Y)\cong\bigoplus_{i=1}^{n}H^2(X_{i}/(X_{i}\cap Y))\cong\bigoplus_{i=1}^{n}H^{2}(X_{i},X_{i}\cap Y).\]
		Let $\iota_{i}\colon H^{2}(X_{i},X_{i}\cap Y)\hookrightarrow H^{2}(X,Y)$ be the natural inclusion.
		Let $E$ be an oriented circle bundle over $X$ with $E|_{Y}$ trivial.
		Let $s$ be a section of $E|_{Y}$. Then we have
		\[e(E,s)=\sum\limits_{i=1}^{n}\iota_{i}(e(E|_{X_{i}},s|_{X_{i}\cap Y})).\]
	\end{Lem}
	
	\begin{proof}
		Let $\pi_{i}\colon H^{2}(X,Y)\twoheadrightarrow H^{2}(X_{i},X_{i}\cap Y) $ be the natural projection.
		We have $\sum\limits_{i=1}^{n}\iota_{i}\pi_{i}=\mathrm{id}$.
		By naturality of the Euler class, $e(E|_{X_{i}},s|_{X_{i}\cap Y})=\pi_{i}(e(E,s))$.
		So,
		\[\sum\limits_{i=1}^{n}\iota_{i}(e(E|_{X_{i}},s|_{X_{i}\cap Y}))=\sum\limits_{i=1}^{n}\iota_{i}(\pi_{i}(e(E,s)))=e(E,s).\]
	\end{proof}
	
	$e(E,s)$ can also be defined from the perspective of obstruction theory.
	Denote the $n$-skeleton of $X$ by $X^{n}$.
	Let $s_{1}$ be an arbitrary extension of $s$ over $X^{1}\cup Y$.
	Such an extension must exist since the fiber $\mathbb{S}^{1}$ is connected.
	The first obstruction occurs when we would like to extend $s_{1}$ over $X^{2}\cup Y$.
	Given a $2$-cell $\sigma\colon e^{2}\rightarrow X$, identify $\sigma^{*}E$ with $e^{2}\times\mathbb{S}^{1}$ by an arbitrary trivialization of $\sigma^{*}E$.
	Then $\sigma^{*}s_{1}|_{\partial e^{2}}$ gives a map $\partial e^{2}\rightarrow\mathbb{S}^{1}$. The assignment of the degree of the map $\partial e^{2}\rightarrow\mathbb{S}^{1}$ to $e^{2}$ for each $2$-cell gives a relative $2$-cochain $c(E,s_{1})\in C^{2}(X,Y)$.
	The relative cochain is actually a relative cocycle, so it descends to a class $[c(E,s_{1})]\in H^{2}(X,Y)$.
	According to the obstruction-theoretic definition of Euler classes, this is exactly the relative Euler class $e(E,s)$.
	In particular, it is independent of the choice of $s_{1}$.
	For more details on this definition, see \cite[\S 4]{Fol2}.
	From this definition, it follows immediately that if $s$ can be extended to a global section, then $e(E,s)=0$, and it is straightforward to obtain the following \textit{restriction formula}.
	
	\begin{Lem}[Restriction formula]\label{Restriction_formula}
		Let $(X,Y)$ be a CW pair, and $Z$ be a subcomplex of $Y$.
		Let $E$ be an oriented circle bundle over $X$ with $E|_{Y}$ trivial.
		Let $s$ be a section of $E|_{Y}$.
		Then we have
		\[e(E,s|_{Z})=\iota(e(E,s)),\]
		where $\iota\colon H^{2}(X,Y)\rightarrow H^{2}(X,Z)$ is the natural homomorphism.
	\end{Lem}
	
	\begin{proof}
		Note that the relative $2$-chain $c(E,s_{1})\in C^{2}(X,Y)$ represents both $e(E,s)\in H^{2}(X,Y)$ and $e(E,s|_{Z})\in H^{2}(X,Z)$.
	\end{proof}
	
	The relative Euler class $e(E,s)$ only depends on the homotopy class of $s$. Let $T(E|_{Y})$ be the set of homotopy classes of sections of $E|_{Y}$. Then $T(E|_{Y})$ is an affine space modeled on $[Y,\mathbb{S}^{1}]\cong H^{1}(Y)$. Denote the homotopy class of $s$ by $[s]_{E}\in T(E|_{Y})$.
	The difference between the relative Euler classes for different sections is given by the following \textit{difference formula}.
    
    \begin{Lem}[Difference formula]\label{Difference_formula}
    	Let $s,s'$ be two sections of $E|_{Y}$.
    	Let $d\colon H^{1}(Y)\rightarrow H^{2}(X,Y)$ be the coboundary map in the cohomology exact sequence of the pair $(X,Y)$.
    	Then
    	\[e(E,s)-e(E,s')=d([s]_{E}-[s']_{E}).\]
    \end{Lem}
    
    \begin{proof}
    	We use the argument in the proof of \cite[Lemma 3.12]{Juh10}. Denote the $n$-skeleton of $Y$ by $Y^{n}$.
    	We can homotope $s,s'$ such that they coincide on $Y^{0}$.
    	Then the cohomology class $[s]_{E}-[s']_{E}$ is represented by the $1$-cocycle $o(s,s')$, whose value on a $1$-cell $e^{1}$ of $Y$ is the homotopy class of $s$ in the trivialization of $E|_{e^{1}}$ given by $s'$.
    	
    	Choose extensions $s_{1},s_{1}'$ of $s,s'$ over $X^{1}\cup Y$ such that they coincide on $X^{1}\backslash Y^{1}$.
    	Given a $2$-cell $\sigma\colon e^{2}\rightarrow X$, $\langle c(E,s_{1})-c(E,s_{2}),e^{2}\rangle$ is the difference between the homotopy classes of the sections $\sigma^{*}s_{1}|_{\partial e^{2}}$ and $\sigma^{*}s_{1}'|_{\partial e^{2}}$ in a trivialization of $\sigma^{*}E|_{e^{2}}$.
    	But $s_{1},s_{1}'$ agree on $X^{1}\backslash Y^{1}$.
    	So we have
    	\[\langle c(E,s_{1})-c(E,s_{2}),e^{2}\rangle=\langle o(s,s'),\partial e^{2}\cap\sigma^{*}Y\rangle.\]
    	This implies that
    	\[e(E,s)-e(E,s')=[c(E,s_{1})]-[c(E,s_{2})]=d([o(s,s')])=d([s]_{E}-[s']_{E}).\]
    \end{proof}
	
	\subsection{Foliations}\label{foliation}
	
	In this subsection, we recall the definition of foliations, paying attention to the regular assumptions. For general facts on foliations, we refer to the textbooks \cite{Fol} and \cite{Fol2}.
	
	\begin{Def}
	    Let $Y$ be a compact oriented $3$-manifold with possibly empty boundary. A \textit{foliated chart} on $Y$ is a pair $(U,\varphi)$ satisfying that:
		
		$\bullet$ $U$ is an open subset in $Y$ such that $\partial U=U\cap\partial Y$;
		
		$\bullet$ $\phi:U\rightarrow B_{2}\times B_{1}$ is a homeomorphism, where $B_{2}$  is $\mathbb{R}^{2}$ or $\mathbb{H}^{2}$, $B_{1}$ is $\mathbb{R}^{1}$ or $\mathbb{H}^{1}$.
		
		For each $y\in B_{1}$, $F_{y}=\varphi^{-1}(F_{y}\times\{y\})$ is called a \textit{plaque} of $U$; $U_{\pitchfork}=\varphi^{-1}(\partial B_{2}\times B_{1})$ is called the \textit{transverse boundary} of $U$; $U_{\tau}=\varphi^{-1}(B_{2}\times\partial B_{1})$ is called the \textit{tangential boundary} of $U$.
	\end{Def}
	
	\begin{Def}
		Let $\mathcal{F}=\{L_{\lambda}\}_{\lambda\in\Lambda}$ be a decomposition of $Y$ into injectively immersed connected surfaces. Suppose that $M$ admits a collection of foliated charts $\{(U_{\alpha},\varphi_{\alpha})\}_{\alpha\in I}$ such that for each $\lambda\in\Lambda$ and $\alpha\in I$, $L_{\lambda}\cap  U_{\alpha}$ is a union of plaques of $U_{\alpha}$. Then $\mathcal{F}$ is called a \textit{topological foliation}. Each $L_{\lambda}$ is called a \textit{leaf} of $\mathcal{F}$. The \textit{transverse boundary} of $Y$ is given by $\partial_{\pitchfork}Y=\cup_{\alpha\in I}\partial_{\pitchfork}U_{\alpha}$, and the \textit{tangential boundary} of $Y$ is given by $\partial_{\tau}Y=\cup_{\alpha\in I}\partial_{\tau}U_{\alpha}$.
	    $\mathcal{F}$ is called \textit{oriented}, if there is a compatible choice of orientations on the leaves.
	\end{Def}
	
	\begin{Def}
		A foliation $\mathcal{F}$ on $Y$ is said to be $C^{\infty,0}$, if each leaf of $\mathcal{F}$ is smooth, and $T\mathcal{F}=\sqcup_{p\in Y}T_{p}L_{p}$ is a $C^{0}$-plane field of $Y$, where $L_{p}$ is the leaf of $\mathcal{F}$ that contains $p$. $T\mathcal{F}$ is called the \textit{tangent bundle} of $\mathcal{F}$.
 	\end{Def}
	
	\begin{Rem}
		The definition of $C^{\infty,0}$ varies in the literature. Based on the results in \cite[\S 5.1]{Fol}, Our definition of a $C^{\infty,0}$ foliation is equivalent to the one in \cite{KR15}, and is slightly stronger than the definition of a $C^{0,\infty+}$-foliation in \cite{BJ16}.
	\end{Rem}
	
	Throughout this paper, all foliations $\mathcal{F}$ are supposed to be oriented and $C^{\infty,0}$. Under the assumption, the \textit{Euler class} of $\mathcal{F}$, denoted by $e(\mathcal{F})$, is defined as the Euler class of the oriented vector bundle $T\mathcal{F}$.

    \begin{Rem}
        The Euler class can be actually defined for an oriented topological foliation $\mathcal{F}$ as follows: Let $L_{p}$ be the leaf of $\mathcal{F}$ that contains $p$. Consider the oriented microbundle
        \[
            Y\xrightarrow{s}E(\mathcal{F})=\sqcup_{p\in Y}L_{p}\xrightarrow{\pi}Y,
        \]
        where $\pi(p,q)=p,\ s(p)=(p,p)$. Let $E^{0}(\mathcal{F})=E(\mathcal{F})\backslash s(Y)$. Then there is a unique class $u\in H^{2}(E(\mathcal{F}),E^{0}(\mathcal{F}))$, called the \textit{Thom class}, such that the restriction to $H^{2}(E(\mathcal{F})|_{p},E^{0}(\mathcal{F})|_{p})\cong\mathbb{Z}$ maps $u$ to the positive generator for each $p\in Y$. Then the \textit{(topological) Euler class} $e_{\mathrm{top}}(\mathcal{F})\in H^{2}(Y)$ is defined as the pullback of $u$ via the map $s\colon (X,\emptyset)\rightarrow (E(\mathcal{F}),E^{0}(\mathcal{F}))$. It is clear that $e_{\mathrm{top}}(\mathcal{F})$ is exactly the Euler class of $T\mathcal{F}$ when $\mathcal{F}$ is $C^{\infty,0}$. It turns out that the Euler class of an oriented foliation is a topological invariant, i.e. it is invariant under the topological isotopies of foliations. By \cite{Calegari01}, any topological foliation is isotopic to a $C^{\infty,0}$ one. So the $C^{\infty,0}$ assumption is inessential, and we impose it only to avoid the language of microbundles and to make the computation more tractable. 
    \end{Rem}
    
	Let $\mathcal{F}$ be a foliation on $Y$. $\mathcal{F}$ is called \textit{Reebless}, if it has no Reeb components and half-Reeb components. For the definition of Reeb components and half-Reeb component, see \cite[Example 1.1.12]{Fol} and \cite[p. 452]{BNR97}, respectively.
	$\mathcal{F}$ is called \textit{taut}, if each leaf meets either a closed transversal or a transverse arc from one component of $\partial_{\tau} Y$ to another.
	A taut foliation is always Reebless, and
    it is well-known that a foliation is taut if it has no annulus leaves or torus leaves \cite[\S 6.3]{Fol}.
	
	\subsection{Sutured manifolds}
	
	Sutured manifolds are originally introduced by Gabai to construct taut foliations on $3$-manifolds \cite{Gab83}. The following discussion is primarily taken from \cite{Yazdi}, and the reader is referred to \cite[\S 2, \S 3]{Yazdi} for a detailed treatment.
	
	\begin{Def}
		A \textit{sutured manifold} $(M,\gamma)$ is a compact oriented $3$-manifold $M$ together with a set $\gamma\subset\partial M$ of pairwise disjoint annuli $A(\gamma)$ and tori $T(\gamma)$. Furthermore, the interior of each component of $A(\gamma)$ contains a \textit{suture}, i.e. a homologically non-trivial oriented simple closed curve. Denote the union of sutures by $s(\gamma)$. Finally, every component of $R(\gamma)=\partial M\backslash\mathrm{int}(\gamma)$ is oriented such that if $\delta$ is a component of $\partial R(\gamma)$ with boundary orientation, then $\delta$ represents the same homology class as some suture in $H_{1}(\gamma)$. Define $R_{+}(\gamma)$ (resp. $R_{-}(\gamma)$) to be the union of the components of $\partial M\backslash\mathrm{int}(\gamma)$ whose normal vectors point out of (resp. into) $M$.
	\end{Def}
	
	\begin{Def}
		An oriented foliation $\mathcal{F}$ on $(M,\gamma)$ is \textit{taut} if the following conditions hold:
		\begin{itemize}
			\item[(1)]
			$R(\gamma)$ is a union of oriented compact leaves of $\mathcal{F}$; each component of $\partial A(\gamma)$ is a convex corner.
			\item[(2)]
			$\mathcal{F}$ is transverse to $\gamma$ and the induced foliation of $\mathcal{F}$ on $\gamma$ contains no Reeb annuli.
			\item[(3)]
			Every leaf of $\mathcal{F}$ has either a closed transversal or a transverse arc starting from $R_{-}(\gamma)$ and ending on $R_{+}(\gamma)$.
		\end{itemize}
		We say that $\mathcal{F}$ is \textit{compatible with the sutured structure} if it satisfies conditions (1) and (2) above.
	\end{Def}
	
	\begin{Rem}
		Condition (2) is equivalent to saying that $\mathcal{F}|_{\gamma}$ is transverse to a fibration of $\gamma$ over $\mathbb{S}^{1}$ \cite[Remark on p. 367]{Fol2}. 
	\end{Rem}
	
	\subsubsection{The relative Euler class using the canonical section}

    Let $\mathcal{F}$ be a foliation on a sutured manifold $(M,\gamma)$ compatible with the sutured structure. In general, the relative Euler class of $\mathcal{F}$ depends on the choice of unit sections of $T\mathcal{F}|_{\partial M}$. However, if $(M,\gamma)$ has toroidal boundary, there is a canonical choice.

    \begin{Def}
        A \textit{sutured manifold with toroidal boundary} is a sutured manifold $(M,\gamma)$ such that $\partial M$ consists of tori, and each component of $s(\gamma)$  is an essential simple closed curve in $\partial M$.
    \end{Def}
	
	Before the description of canonical sections, We state the fact that the tangent bundle of a $2$-dimensional torus has a canonical unit section up to homotopy \cite[p. 338]{Yazdi}. We refer to such a canonical section as a \textit{Lie group section}.
	
	\begin{Def}[{\cite[Definition 3.9]{Yazdi}}]
		Let $(M,\gamma)$ be a sutured manifold with toroidal boundary. Let $\mathcal{F}$ be a foliation on $(M,\gamma)$ compatible with the sutured structure. A unit section $s$ of $T\mathcal{F}|_{\partial M}$ is called a \textit{canonical section}, if it satisfies the following:
		\begin{itemize}
			\item[(1)]
			$s|_{\gamma}$ is normal to $\partial M$.
			\item[(2)]
			$s$ is a Lie group section along each torus component of $R(\gamma)$.
			\item[(3)] For each annulus component $A$ of $R(\gamma)$, split $A$ as a product $\mathbb{S}^{1}\times I$. Then $s$ is transverse to the $\mathbb{S}^{1}$-fibers.
		\end{itemize}
	\end{Def}
	
	\begin{Rem}
		The canonical section is well-defined up to homotopy.
	\end{Rem}
	
	\begin{Def}
		Let $(M,\gamma)$ be a sutured manifold with toroidal boundary. Let $\mathcal{F}$ be a foliation on $(M,\gamma)$ compatible with the sutured structure. The \textit{relative Euler class} of $\mathcal{F}$, denoted by $e(\mathcal{F},\partial)$, is defined as $e(T\mathcal{F},s)\in H^{2}(M,\partial M)$, where $s$ is a canonical trivialization of $T\mathcal{F}|_{\partial M}$.
	\end{Def}
	
	\subsubsection{The parity condition and Thurston's inequality}
	
	Let $Y$ be a closed oriented $3$-manifold, and $S$ be a closed oriented surface in $Y$. It is well-known that the Euler class of a foliation $\mathcal{F}$ on $Y$ satisfies the \textit{parity condition}
	\[\langle e(\mathcal{F}),[S]\rangle\equiv0\mod 2.\]
    If $\mathcal{F}$ is $C^{\infty}$-Reebless and $S$ has no spherical components, then Thurston proves the adjunction inequality
	\[|\langle e(\mathcal{F}),[S]\rangle|\leq-\chi(S),\]
	known as the \textit{Thurston's inequality} \cite{Th86}. Both the parity condition and Thurston's inequality can be generalized to $C^{0}$ foliations on $3$-manifolds with toroidal boundaries.
	
	\begin{Lem}\label{Adjunction_inequality}
		Let $(M,\gamma)$ be a sutured manifold with toroidal boundary such that $A(\gamma)=\emptyset$. Let $\mathcal{F}$ be a foliation on $(M,\gamma)$ compatible with the sutured structure. Let $S$ be a compact oriented surface in $M$ such that $\partial S$ consists of essential simple closed curves in $\partial M$. Then we have
		\begin{itemize}
			\item [(1)]\textnormal{(Parity condition)}
			\[\langle e(\mathcal{F},\partial),[S]\rangle\equiv\chi(S)\mod 2.\]
			\item [(2)]\textnormal{(Thurston's inequality)}
			If $\mathcal{F}$ is Reebless and $S$ has no spherical or disk components, then
			\[|\langle e(\mathcal{F},\partial),[S]\rangle|\leq-\chi(S).\]
		\end{itemize}
	\end{Lem}
	
	\begin{proof}
		(1) Consider the double $DM=M\cup_{\partial M}\bar{M}$. Regard $M$ as an oriented submanifold of $DM$. Denote by $\sigma$ the obvious involution on $DM$. The double $DS=S\cup_{\partial S}\sigma(\bar{S})$ of $S$ is a closed oriented surface in $DM$.
		
	    Let $s$ be a canonical section of $T\mathcal{F}|_{\partial M}$. Then $T\mathcal{F}|_{S}$ can be homotopic, as an oriented subbundle of $TM|_{S}$, to another subbundle $V$ over $S$, such that $V$ and $TS$ agree near $\partial S$ and $s|_{\partial S}$ is sent to the normal vector field $\nu$ of $\partial S\subset S$. We have $e(\mathcal{F},\partial)=e(\mathcal{F},s)=e(V,\nu)$. The union of $V$ over $S$ and the tangent bundle $T(\sigma(\bar{S}))$ over $\sigma(\bar{S})$ form an oriented plane field $W$ over $DS$.
		We have
		\[\langle e(W),[DS]\rangle=\langle e(V,\nu),[S]\rangle+\langle e(T(\sigma(\bar{S})),\nu),[\sigma(\bar{S})]\rangle=e(V,\nu)+\chi(S).\]
		On the other hand, $ W\oplus\underline{\mathbb{R}}\cong T(DM)|_{DS}\cong\underline{\mathbb{R}}^{3}$. So $e(W)\mod 2=w_{2}(W)=0$. So $\langle e(W),[DS]\rangle\equiv 0\mod 2$ and it follows that $\langle e(\mathcal{F},\partial),[S]\rangle\equiv\chi(S)\mod 2$.

        (2) We first consider the case $\partial_{\pitchfork}M=\emptyset$. Let $P_{1},\ldots,P_{n}$ be all boundary components of $M$. Choose the orientation of $P_{i}$ to coincide with the orientation of the leaves of $\mathcal{F}$. We can suppose that for each $i$, $\partial S\cap P_{i}$ consists of $m_{i}$ essential simple closed curves, and each of them represents the same class $\alpha_{i}$ in $H_{1}(P_{i})$. For each $i$, pick a compact oriented surface $S_{i}$ of positive genus with a single boundary. Let $M_{i}=S_{i}\times\mathbb{S}^{1}$. Choose an orientation-reversing diffeomorphism $\varphi_{i}\colon\partial M_{i}=\partial S_{i}\times\mathbb{S}^{1}\rightarrow P_{i}$ such that if $P_{i}\cap\partial S$ is non-empty, then $\varphi_{i}(\partial S_{i}\times\theta)$ represents the class $-\alpha_{i}\in H_{1}(P_{i})$. Let $\hat{M}$ be the closed oriented $3$-manifold obtained by gluing $\sqcup_{1\leq i\leq n}M_{i}$ to $M$ via $\sqcup_{1\leq i\leq n}\varphi_{i}\colon\sqcup_{1\leq i\leq n}\partial M_{i}\rightarrow\partial M$ and let $\hat{S}$ be the closed oriented surface in $\hat{M}$ obtained from $S$ by capping each component of $\partial S\cap P_{i}$ with a copy of $S_{i}$. Then $\hat{S}$ has no spherical components, and we have
        \begin{equation}\label{chi_equality}
            \chi(\hat{S})=\chi(S)+\sum\limits_{i=1}^{n}m_{i}\chi(S_{i}).
        \end{equation}

        We construct a foliation $\mathcal{F}_{i}$ on $M_{i}=S_{i}\times\mathbb{S}^{1}$ as follows. Choose a collar neighborhood $C_{i}\cong \partial S\times [0,1)$ with coordinates $(\psi,t)$. Fix a smooth function $h\colon (0,1)\rightarrow\mathbb{R}$ such that $h(t)=0$ for $t\geq 1/2$ and $\lim\limits_{t\rightarrow 0}h(t)=+\infty$. Define the smooth function $H_{i}\colon\mathrm{int}(S_{i})\rightarrow\mathbb{R}$ by letting $H_{i}(\psi,t)=h(t)$ in $C$ and $H_{i}=0$ outside $C$. The interior leaves of $\mathcal{F}'$ are given by $L_{\theta}=\{(x,\phi)\colon\phi=\theta+H_{i}(x)\mod 2\pi\}$, where we identify $\mathbb{S}^{1}$ with $\mathbb{R}/2\pi\mathbb{Z}$. Finally, add the boundary leaf $L=\partial S_{i}\times\mathbb{S}^{1}$ to $\mathcal{F}_{i}$. It can be checked that the orientation of $L$ is opposite to the boundary orientation of $\partial M_{i}$, thus it agrees with the orientation of $P_{i}$ in $\hat{M}$.
        So $\mathcal{F}$ and $\mathcal{F}_{i}\ (1\leq i\leq n)$ together form a Reebless foliation $\hat{\mathcal{F}}$ on $\hat{M}$.
        By \cite[Theorem 9.8]{BJ16}, $T\hat{\mathcal{F}}$ is homotopic to a tight contact structure $\xi$. It is well-known that $e(\xi)$ satisfies the adjunction inequality
        $\langle e(\xi),[\hat{S}]\rangle\geq\chi(\hat{S})$
        \cite[Theorem 2.2.1]{Elia92}. So we have
		\begin{equation}\label{adjun_inequality}
		    \langle e(\hat{\mathcal{F}}),[\hat{S}]\rangle\geq\chi(\hat{S}).
		\end{equation}
        
        Let $s_{i}$ be a canonical section of $T\mathcal{F}_{i}|_{\partial M_{i}}$. Identify $S_{i}$ with $S_{i}\times 1\subset M_{i}$. Then $T\mathcal{F}_{i}|_{S_{i}}$ can be homotopic, as an oriented subbundle of $TM_{i}|_{S_{i}}$, to $TS_{i}$ so that $s_{i}|_{\partial S_{i}}$ is sent to the normal vector field of $\partial S_{i}\subset S_{i}$. So we have $\langle e(\mathcal{F}_{i},\partial),[S_{i}]\rangle=\chi(S_{i})$. Note that $[\hat{S}\cap M_{i}]=m_{i}[S_{i}]$ in $H_{2}(M_{i},\partial M_{i})$. So we have
        \[
            \langle e(\hat{\mathcal{F}}),[\hat{S}]\rangle=\langle e(\mathcal{F},\partial),[S]\rangle+\sum\limits_{i=1}^{n}\langle e(\mathcal{F}_{i},\partial),m_{i}[S_{i}])\rangle=\langle e(\mathcal{F},\partial),[S]\rangle+\sum\limits_{i=1}^{n}m_{i}\chi(S_{i}).
        \]
        This equality together with (\ref{chi_equality}) and (\ref{adjun_inequality}) imply that $\langle e(\mathcal{F},\partial),[S]\rangle\geq\chi(S)$. Replacing $S$ by $\bar{S}$, we have $-\langle e(\mathcal{F},\partial),[S]\rangle\geq\chi(S)$. We deduce that
        \[
            |\langle e(\mathcal{F},\partial),[S]\rangle|\leq-\chi(S).
        \]
        
        In general, consider the double $D_{\pitchfork}M=M\cup_{\partial_{\pitchfork}M}\bar{M}$ along $\partial_{\pitchfork}M$ with the involution $\sigma$. Then $S$ doubles to a embedded oriented surface $D_{\pitchfork}S=S\cup\sigma(\bar{S})$ in $D_{\pitchfork}M$ and $\mathcal{F}$ doubles to an oriented foliation $D_{\pitchfork}\mathcal{F}=\mathcal{F}\cup\sigma(\bar{\mathcal{F}})$ on $D_{\pitchfork}M$.
        
        We claim that $D_{\pitchfork}\mathcal{F}$ is Reebless. Indeed, suppose that $R$ is a Reeb component of $D_{\pitchfork}\mathcal{F}$. If $\partial R$ lies entirely in $M$ or $\sigma(M)$, we can suppose that $\partial R\subset \mathrm{int}(M)$. Since $\mathcal{F}$ is Reebless, $R$ cannot lie entirely in $M$. So $\sigma(M)\subset\mathrm{int}(R)$. But then $\sigma(\partial R)$ would be a torus leaf in $\mathrm{int}(R)$, which is impossible. So $\partial R$ must meet some component $T$ of $\partial_{\pitchfork}M$. Since $\mathcal{F}|_{T}$ contains no Reeb annuli, it follows that every $1$-dimensional leaf of $\mathcal{F}|_{T}$ has a closed transversal on $T$. However, $\partial R$ has no closed transversal. This also leads to a contradiction. So $D_{\pitchfork}\mathcal{F}$ is Reebless.
        
        Note that $D_{\pitchfork}\mathcal{F}$ is tangent to $\partial(D_{\pitchfork}M)$. The previous discussion shows that
        \[
            |\langle e(D_{\pitchfork}\mathcal{F},\partial),[D_{\pitchfork}S]\rangle|\leq-\chi(D_{\pitchfork}S).
        \]
        Besides, we have
        \[
            \langle e(D_{\pitchfork}\mathcal{F},\partial),[D_{\pitchfork}S]\rangle=\langle e(\mathcal{F},\partial),[S]\rangle+\langle e(\sigma(\bar{\mathcal{F}}),\partial),[\sigma(\bar{S})]\rangle=2\langle e(\mathcal{F},\partial),[S]\rangle.
        \]
        Consequently,
        \[
            |\langle e(\mathcal{F},\partial),[S]\rangle|=\frac{1}{2}|\langle e(D_{\pitchfork}\mathcal{F},\partial),[D_{\pitchfork}S]\rangle|\leq-\frac{1}{2}\chi(D_{\pitchfork}S)=-\chi(S).
        \]
	\end{proof}
	
	\begin{Rem}
	    If $\mathcal{F}$ is at least $C^{2}$, then Thurston's approach is to perturb $S$ so that $S$ is transverse to the leaves of $\mathcal{F}$ except at finitely many saddle or center tangencies. If moreover $\mathcal{F}$ is Reebless and $S$ is incompressible, then $S$ can be further isotoped so that $S$ is transverse to the leaves of $\mathcal{F}$ except at finitely many saddle or circle tangencies. Once we place $S$ in such a nice position, both parity condition and Thurston's inequality follow from a simple index calculation. However, It is not clear to the author whether the argument can adapted to $C^{0}$-case.
        So we adopt an indirect approach, reducing the problem into the closed manifold case, which we know how to treat. 
	\end{Rem}

    \begin{Rem}
        Be aware that if both $\partial_{\pitchfork}M$ and $\partial_{\tau}M$ are non-empty, then the doubled foliation $D\mathcal{F}$ on $DM$ is non-orientable. Besides, if we double $M$ along $\partial_{\tau}M$, then $D_{\tau}\mathcal{F}=\mathcal{F}\cup\sigma(\mathcal{F})$ as oriented foliations, and it follows that $\langle e(D_{\tau}\mathcal{F}),[D_{\tau}S]\rangle=0$. So we choose to double $M$ along $\partial_{\pitchfork}M$ in the proof of (2).
    \end{Rem}
	
	\section{Standard laminations on compact surfaces}\label{Standard_lamination_on_surface}
	
	In this section, we assign an oriented essential lamination to each compact orientable surface, which we would use to construct laminations on JSJ pieces of graph manifolds.
	
	To begin with, let $\Sigma_{2}$ be a genus $2$ closed orientable surface. Choose a $2$-fold branched cover $p\colon\Sigma_{2}\rightarrow\mathbb{T}^{2}$. Let $f\colon\mathbb{T}^2\rightarrow \mathbb{T}^2$ be an arbitrary Anosov linear automorphism. Then $f$ lifts to a pseudo-Anosov map $f_{2}\colon\Sigma_{2}\rightarrow\Sigma_{2}$ with an orientable stable essential lamination $\lambda_{2}$. Equip $\lambda_{2}$ with an arbitrary orientation. The complementary region of $\lambda_{2}$ consists of two $4$-gons. 
	
	Let $\Sigma$ be a closed orientable surface of genus $g\geq1$. If $\Sigma$ is a torus, then we call any oriented linear foliation $\lambda$ with irrational slope on $\Sigma$ a \textit{standard essential lamination} on $\Sigma$. If $g\geq 2$, then choose an arbitrary $(g-1)$-fold cover $\pi\colon\Sigma\rightarrow\Sigma_{2}$. Let $\lambda$ be the pull-back of $\lambda_{2}$ to $\Sigma$. The oriented essential lamination $\lambda$ is called a \textit{standard lamination} on $\Sigma$. Note that in either case the complementary region of $\lambda$ consists of $2(g-1)$ $4$-gons.
	
	More generally, let $\Sigma$ be a compact orientable surface of genus $g\geq 1$ with $k$ boundaries. Fill a disk along each boundary component to obtain a closed surface $\hat{\Sigma}$. Let $\hat{\lambda}$ be a standard essential lamination of $\hat{\Sigma}$. Open up $k$ leaves of $\hat{\lambda}$ to obtain a new lamination $\lambda$ whose complementary region consists of $2(g-1)$ $4$-gons and $k$ $2$-gons. After an isotopy, we can suppose that each $2$-gon contains exactly one filling disk. Then $\lambda$ can be regarded as an oriented essential lamination on $\Sigma$. We call $\lambda$ a \textit{standard lamination} on $\Sigma$. Note that the complementary region of $\lambda$ consists of $2(g-1)$ $4$-gons and $k$ holed $2$-gons.
	
	\section{Taut foliations on sutured solid tori and sutured torus shells}\label{SST}
	
	Let $(M,\gamma)$ be a sutured manifold. $(M,\gamma)$ is called a \textit{sutured solid torus}, if $M$ is topologically a solid torus and $\gamma$ consists of $2m$ parallel annuli on the boundary for some $m\geq 1$; $(M,\gamma)$ is called a \textit{half sutured torus shell}, if $M$ is topologically a $\mathbb{T}^{2}\times[0,1]$, $T(\gamma)=\mathbb{T}^{2}\times\{0\}$ and $A(\gamma)$ consists of $2m$ parallel annuli on $\mathbb{T}^{2}\times\{1\}$ for some $m\geq 1$; $(M,\gamma)$ is called a \textit{sutured torus shell}, if $M$ is topologically a $\mathbb{T}^{2}\times[0,1]$, and $\gamma$ consists of $2m_{0}$ parallel annuli on $\mathbb{T}^{2}\times\{0\}$ together with $2m_{1}$ parallel annuli on $\mathbb{T}^{2}\times\{1\}$ for some $m_{0},m_{1}\geq 1$.
	
	It is well-known that taut solid tori can be tautly foliated by stacks of saddles (see \cite[\S 3.7]{Yazdi}, see also \cite[\S 11.1]{Fol2}). The construction can also yield taut foliations on sutured torus shells. In this section, we give an explicit description of these foliations and calculate their relative Euler classes.
	
	We begin with the standard model for half sutured torus shells. We use the coordinate $(x,y,h)$ for $\mathbb{R}^{3}$. Also write $z=x+y\sqrt{-1}$ for the first two coordinates. Let $\mathbb{D}$ be the closed unit disk in $\mathbb{C}$. Identify the circle $\mathbb{S}^{1}$ with $\mathbb{R}/2\pi\mathbb{Z}$. Choose an arbitrary smooth increasing function $\phi\colon(-2,2)\rightarrow\mathbb{R}$ such that
	\begin{enumerate}
	    \item $\lim\limits_{y\rightarrow\pm2}\phi(y)=\pm\infty$.
        \item $\phi(y)=0$ in a neighborhood of $[-1,1]$.
	\end{enumerate}
	Let
	$\tilde{N}=([-2,2]\times [-2,2]\backslash\mathrm{int}(\mathbb{D}))\times\mathbb{R},\ \tilde{R}_{\pm}=[-2,2]\times\{\pm2\}\times\mathbb{R},\ \tilde{\gamma}=\partial\tilde{N}\backslash\mathrm{int}(\tilde{R}_{+}\cup\tilde{R}_{-})$.
	Consider the foliations $\tilde{\mathcal{F}}^{\pm}$ on $\tilde{N}$ with the \textit{interior leaves} $\{\tilde{F}^{\pm}(t)\}_{t\in\mathbb{R}}$ given by
	\[
        \tilde{F}^{\pm}(t)=\{(x,y,h)\in \tilde{N}\colon h=\pm\phi(y)+t\},
    \]
	and the \textit{boundary leaves} $\tilde{R}_{+}$ and $\tilde{R}_{-}$. Choose the orientation of $\tilde{\mathcal{F}}^{\pm}$ such that the normal vectors of $\tilde{R}_{+}$ (resp. $\tilde{R}_{-}$) (with respect to the leaf orientations) point out of (resp. into) $\tilde{N}$.
	
	Let $m,p,q$ be integers with $m,p>0$ and $\gcd(p,q)=1$.
	Consider the $mp$-fold cover $\pi_{m,p,q}\colon(\mathbb{C}\backslash \mathrm{int}(\mathbb{D}))\times\mathbb{R}\rightarrow(\mathbb{C}\backslash \mathrm{int}(\mathbb{D}))\times\mathbb{R}$ given by
	\[
       \pi_{m,p,q}(z,h)=(e^{-mq h\sqrt{-1}}z^{mp},h).
    \]
	Let $\tilde{N}_{m,p,q}, \tilde{R}_{m,p,q,\pm}, \tilde{\gamma}_{m,p,q}, \tilde{\mathcal{F}}_{m,p,q}^{\pm}$ be the pull-back objects of $\tilde{N}, \tilde{R}_{\pm}, \tilde{\gamma}, \tilde{\mathcal{F}}^{\pm}$ via $\pi_{m,p,q}$, respectively. Note that these objects are invariant under the translation
	\[
       \tau(z,h)=(z,h+2\pi).
    \]
	Let $N_{m,p,q}, R_{m,p,q,\pm}, \gamma_{m,p,q}, \mathcal{F}_{m,p,q}^{\pm}$ be the quotients of $\tilde{N}_{m,p,q}, \tilde{R}_{m,p,q,\pm}, \gamma_{m,p,q}, \tilde{\mathcal{F}}_{m,p,q}^{\pm}$ by the $\mathbb{Z}$-action generated by $\tau$, respectively.
	
	It turns out that $(N_{m,p,q},\gamma_{m,p,q})$ is a half sutured torus shell with
    \[
        R_{\pm}(\gamma_{m,p,q})=R_{m,p,q,\pm},\quad T(\gamma_{m,p,q})=\partial\mathbb{D}\times\mathbb{S}^{1}.
    \]
	$\mathcal{F}_{m,p,q}^{\pm}$ are taut foliations on $(N_{m,p,q},\gamma_{m,p,q})$ with the \textit{interior leaves} $\{F_{m,p,q}^{\pm}(\theta)\}_{\theta\in\mathbb{S}^{1}}$ given by
	\[
        F_{m,p,q}^{\pm}(\theta)=\{(z,h)\in N_{m,p,q}\colon h=\pm\phi(\mathrm{Im}(e^{-mq h\sqrt{-1}}z^{mp}))+\theta\}.
    \]
    Let $T=\partial\mathbb{D}\times\mathbb{S}^{1}$, $\mu=\{(e^{i\theta_{1}},0)\colon\theta_{1}\in\mathbb{S}^{1}\}$ and $\lambda=\{(1,\theta_{2})\colon\theta_{2}\in\mathbb{S}^{1}\}$. Orient $\mu$ and $\lambda$ by the tangent fields $\partial_{\theta_{1}}$ and $\partial_{\theta_{2}}$, respectively. Then $N_{m,p,q}\cong T\times[0,1]$ and we have $T(\gamma_{m,p,q})=T\times\{0\}$ and $s(\gamma_{m,p,q})$ consists of $2m$ simple closed curves of slope $q[\mu]+p[\lambda]$ on $T\times\{1\}$ under the isomorphism.
    
    \begin{Lem}\label{Euler_class_for_standard_model}
    	The relative Euler class of $\mathcal{F}_{m,p,q}^{\pm}$ is given by
    	\[
            e(\mathcal{F}_{m,p,q}^{\pm},\partial)=\pm PD(mq[\mu]+mp[\lambda]).
        \]
    \end{Lem}
    
    \begin{proof}
    	Let $V\rightarrow T$ be the subbundle of $T\mathbb{R}^{3}|_{T}$ generated by $\partial_{x}$ and $\partial_{y}$. We can naturally identify $V$ with the trivial complex line bundle $\underline{\mathbb{C}}$ over $T$. Choose the orientation of $V$ to coincide with the complex orientation. It can be checked that $T\mathcal{F}_{m,p,q}^{-}|_{T}=V$ and $T\mathcal{F}_{m,p,q}^{+}|_{T}=\bar{V}$.
    	
    	Note that $\partial_{x}$ is tangent to the leaves of $\tilde{\mathcal{F}}^{\pm}$.  So $\pi_{m,p,q}^{*}\partial_{x}$ can be normalized to a unit section $s$ of both $T\mathcal{F}_{m,p,q}^{\pm}$.
    	Then we have
    	\[
            e(T\mathcal{F}_{m,p,q}^{\pm},s)=0.
        \]
    	Here we denote $s|_{\partial N_{m,p,q}}$ also by $s$ for simplicity. Let $s_{0}$ be a canonical section of $T\mathcal{F}_{m,p,q}^{\pm}|_{\partial N_{m,p,q}}$.
    	We can suppose that $s_{0}|_{T}$ is given by
        \[
            s_{0}|_{(e^{\theta_{1}\sqrt{-1}},\theta_{2})}=e^{\theta_{1}\sqrt{-1}}.
        \]
    	It can be checked that $s|_{\partial N_{m,p,q}\backslash T}$ is homotopic to $s_{0}|_{\partial N_{m,p,q}\backslash T}$.
    	So we have
    	\begin{align*}
    		e(\mathcal{F}_{m,p,q}^{\pm},\partial)&=e(T\mathcal{F}_{m,p,q}^{\pm},s_{0})-e(T\mathcal{F}_{m,p,q}^{\pm},s)\\
    		&=d([s_{0}]_{T\mathcal{F}_{m,p,q}^{\pm}}-[s]_{T\mathcal{F}_{m,p,q}^{\pm}})\\
    		&=d([s_{0}|_{T}]_{T\mathcal{F}_{m,p,q}^{\pm}}-[s|_{T}]_{T\mathcal{F}_{m,p,q}^{\pm}})\\
    		&=\mp d([s_{0}|_{T}]_{V}-[s|_{T}]_{V}),
    	\end{align*}
    	where $d\colon H^{1}(\partial N_{m,p,q})\rightarrow H^{2}(N_{m,p,q},\partial N_{m,p,q})$ is the coboundary map for the pair $(N_{m,p,q},\partial N_{m,p,q})$.
    	
    	Choose the branch of the multivalued function $z^{1/mp}$ near $z_{0}=e^{mp\theta\sqrt{-1}}$ such that $z_{0}^{1/mp}=e^{\theta\sqrt{-1}}$.
    	Then we have
    	\begin{align*}
    		\pi_{m,p,q}^{*}(\partial_{x})|_{(e^{\theta\sqrt{-1}},h)}=&\left.\dfrac{\mathrm{d}}{\mathrm{d}t}\right|_{t=0}[e^{mqh\sqrt{-1}}(e^{-mqh\sqrt{-1}}e^{mp\theta\sqrt{-1}}+t)]^{1/mp}\\
    		=&\left.\dfrac{\mathrm{d}}{\mathrm{d}t}\right|_{t=0}(z_{0}+e^{mqh\sqrt{-1}}t)^{1/mp}\\
    		=&\dfrac{1}{mp}e^{mqh\sqrt{-1}}z_{0}^{-1+1/mp}\\
    		=&\dfrac{1}{mp}e^{((1-mp)\theta+mqh)\sqrt{-1}}.
    	\end{align*} 
        So we have
        \[
            s|_{(e^{\theta_{1}\sqrt{-1}},\theta_{2})}=e^{((1-mp)\theta_{1}+mq\theta_{2})\sqrt{-1}}.
        \]
        It follows that $[s_{0}|_{T}]_{V}-[s|_{T}]_{V}\in H^{1}(T)$ is represented by the map $f\colon T\rightarrow\mathbb{S}^{1}$ given by $f(e^{\theta_{1}\sqrt{-1}},\theta_{2})=mp\theta_{1}-mq\theta_{2}$.
        
        Let $A_{1}=[0,1]\times\mu$ and $A_{2}=[0,1]\times\lambda$ be the properly embedded annuli in $N_{m,p,q}$ with the product orientations. Then $H_{2}(N_{m,p,q},\partial N_{m,p,q})\cong\mathbb{Z}^{2}$ is generated by $[A_{1}]$ and $[A_{2}]$.
        We have
        \begin{align*}
        	\langle d([f]),[A_{1}]\rangle&=\langle[f],[\partial A_{1}\cap T]\rangle=\langle[f],-[\mu]\rangle=-\deg(f|_{\mu})=-mp\\
        	&=-\langle PD(mq[\mu]+mp[\lambda]),[A_{1}]\rangle,\\
        	\langle d([f]),[A_{2}]\rangle&=\langle[f],[\partial A_{2}\cap T]\rangle=\langle[f],-[\lambda]\rangle=-\deg(f|_{\lambda})=mq\\
        	&=-\langle PD(mq[\mu]+mp[\lambda]),[A_{2}]\rangle.
        \end{align*}
        Since $H^{2}(N_{m,p,q},\partial N_{m,p,q})\cong H_{1}(N_{m,p,q})$ is torsion-free, it follows that
        \[
            d([f])=-PD(mq[\mu]+mp[\lambda]).
        \]
        Therefore
        \[
            e(\mathcal{F}_{m,p,q}^{\pm},\partial)=\mp d([s_{0}|_{T}]_{V}-[s|_{T}]_{V})=\mp d([f])=\pm PD(mq[\mu]+mp[\lambda]).
        \]
    \end{proof}
	
	\begin{Lem}\label{Foliation_on_sutured_torus_shell}
		Let $\mathbb{T}$ be an oriented torus. Let $m_{0},m_{1}$ be positive integers. Let $\alpha_{0},\alpha_{1}$ be primitive classes in $H_{1}(\mathbb{T})$ with $\alpha_{0}\not=\pm\alpha_{1}$. Let $(M,\gamma)$ be the sutured torus shell such that $M=\mathbb{T}\times [0,1]$ and $s(\gamma)$ consists of $2m_{0}$ simple closed curves of slope $[\alpha_{0}]$ on $\mathbb{T}\times\{0\}$ and $2m_{1}$ simple closed curves of slope $[\alpha_{1}]$ on $\mathbb{T}\times\{1\}$. Then for any $\epsilon_{0},\epsilon_{1}\in\{\pm1\}$, there is a taut foliation $\mathcal{F}$ on $(M,\gamma)$ such that
		\[
            e(\mathcal{F},\partial)=PD(\epsilon_{0}m_{0}\alpha_{0}+\epsilon_{1}m_{1}\alpha_{1}).
        \]
	\end{Lem}
	
	\begin{proof}
	    Let $Q_{\mathbb{T}}\colon H_{1}(\mathbb{T})\times H_{1}(\mathbb{T})\rightarrow\mathbb{Z}$ be the intersection form. Choose a primitive class $\alpha\in H_{1}(\mathbb{T})$ such that $\epsilon_{0}Q_{\mathbb{T}}(\alpha,\alpha_{0})>0$ and $\epsilon_{1}Q_{\mathbb{T}}(\alpha,\alpha_{1})>0$, and a primitive class $\beta\in H_{1}(\mathbb{T})$ with $Q_{\mathbb{T}^{2}}(\alpha,\beta)=1$.
		Then $\epsilon_{0}\alpha_{0}=p_{0}\beta+q_{0}\alpha$ and $\epsilon_{1}\alpha_{1}=p_{1}\beta+q_{1}\alpha$ for some integer $p_{0},q_{0},p_{1},q_{1}$ with $p_{0},p_{1}>0$.
		Consider the sutured submanifolds
        \begin{align*}
            (M_{0},\gamma_{0})&=(\mathbb{T}\times[0,1/2],(\gamma\cap M_{0})\cup(\mathbb{T}\times\{1/2\})),\\
            (M_{1},\gamma_{1})&=(\mathbb{T}\times[1/2,1],(\gamma\cap M_{1})\cup(\mathbb{T}\times\{1/2\})).
        \end{align*}
		Then there is an isomorphism $\Phi_{0}\colon(M_{0},\gamma_{0})\xrightarrow{\cong}(N_{m_{0},p_{0},-q_{0}})$ such that $(\Phi_{0})_{*}(\alpha)=-[\mu]$ and $(\Phi_{0})_{*}(\beta)=[\lambda]$,
		and we obtain a taut foliation $\mathcal{F}_{0}=\Phi_{0}^{*}\mathcal{F}_{m_{0},p_{0},-q_{0}}^{+}$
		on $(M_{0},\gamma_{0})$.
		By Lemma \ref{Euler_class_for_standard_model}, we have
		\begin{align*}
			e(\mathcal{F}_{0},\partial)
			&=\Phi_{0}^{*}(e(\mathcal{F}_{m_{0},p_{0},-q_{0}}^{+},\partial))\\
			&=\Phi_{0}^{*}(PD(-m_{0}q_{0}[\mu]+m_{0}p_{0}[\lambda]))\\
			&=PD(m_{0}q_{0}\alpha+m_{0}p_{0}\beta)\\
			&=PD(\epsilon_{0}m_{0}\alpha_{0}).
		\end{align*}
		Similarly, there is an isomorphism $\Phi_{1}\colon(M_{1},\gamma_{1})\xrightarrow{\cong}(N_{m_{1},p_{1},q_{1}})$ such that $(\Phi_{1})_{*}(\alpha)=[\mu]$ and $(\Phi_{1})_{*}(\beta)=[\lambda]$,\and we obtain a taut foliation $\mathcal{F}_{1}=\Phi_{1}^{*}\mathcal{F}_{m_{1},p_{1},q_{1}}^{+}$
		  on $(M_{1},\gamma_{1})$.
		  By Lemma \ref{Euler_class_for_standard_model}, we have
		  \begin{align*}
		 	  e(\mathcal{F}_{1},\partial)
		 	&=\Phi_{1}^{*}(e(\mathcal{F}_{m_{1},p_{1},q_{1}}^{+},\partial))\\
		 	&=\Phi_{1}^{*}(PD(m_{1}q_{1}[\mu]+m_{1}p_{1}[\lambda]))\\
		 	&=PD(m_{1}q_{1}\alpha+m_{1}p_{1}\beta)\\
		 	&=PD(\epsilon_{1}m_{1}\alpha_{1}).
		  \end{align*}
		  Both $\mathcal{F}_{0}|_{\mathbb{T}\times\{1/2\}}$ and
        $\mathcal{F}_{1}| _{\mathbb{T}\times\{1/2\}}$ consists of simple closed curves of slope $[\alpha]$, so after an isotopy we can suppose that $\mathcal{F}_{0}|_{\mathbb{T}\times\{1/2\}}=\mathcal{F}_{1}|_{\mathbb{T}\times\{1/2\}}$. Moreover, the orientations of $\mathcal{F}_{0}$ and $\mathcal{F}_{1}$ are compatible along $\mathbb{T}\times\{1/2\}$. So $\mathcal{F}_{0}$ and $\mathcal{F}_{1}$ together form a taut foliation $\mathcal{F}$ on $(M,\gamma)$. It follows from the gluing formula (Lemma \ref{Gluing_formula}) and the restriction formula (Lemma \ref{Restriction_formula}) for relative Euler classes  that
		  \[
            e(\mathcal{F},\partial)=PD(\epsilon_{0}m_{0}\alpha_{0}+\epsilon_{1}m_{1}\alpha_{1}).
        \]
	\end{proof}
	
	\begin{Lem}\label{Foliation_On_sutured_solid_torus}
		Let $(M,\gamma)$ be a sutured solid torus. Let $\mu$ be a meridian of $\partial M$ and let $\lambda$ be a longitude of $\partial M$. If $s(\gamma)$ consists of $2m$ sutures each of which goes $p$ times $(p>0)$ around $\lambda$ and $q$ times around $\mu$, then there are taut foliations $\mathcal{F}^{\pm}$ on $(M,\gamma)$ such that
		\[
            e(\mathcal{F}^{\pm},\partial)=\pm(mp-1)PD[\lambda].
        \]
	\end{Lem}
	
	\begin{proof}
		Up to isomorphism, we can suppose that $M=N_{m,p,q}\cup_{T}(\mathbb{D}\times\mathbb{S}^{1})$ and $\gamma=\gamma_{m,p,q}\backslash T$. 
        Let $\mathcal{P}^{\pm}$ denote the product foliation on $\mathbb{D}\times\mathbb{S}^{1}$ foliated by $\mathbb{D}\times p \ (p\in\mathbb{S}^{1})$, whose orientation is given by $\mp \mathrm{d} x\wedge \mathrm{d} y$. Then $\mathcal{P}^{\pm}|_{T}=\mathcal{F}_{m,p,q}^{\pm}|_{T}$ and their orientations are compatible along $T$. So $\mathcal{P}^{\pm}$ and $\mathcal{F}_{m,p,q}^{\pm}$ together form a taut foliation $\mathcal{F}^{\pm}$ on $(M,\gamma)$.
        
        Let $s_{0}$ be a canonical section of $T\mathcal{F}_{m,p,q}^{+}|_{\partial N_{m,p,q}}$ with $s_{0}|_{(e^{\theta_{1}\sqrt{-1}},\theta_{2})}=e^{\theta_{1}\sqrt{-1}}$. Since $\partial_{x}$ gives a unit section of $T\mathcal{P}^{\pm}$, we have $e(T\mathcal{P}^{\pm},\partial_{x}|_{T})=0$.
		So
		\[
            e(\mathcal{P}^{\pm},s_{0}|_{T})=d([s_{0}]_{T\mathcal{F}^{\pm}}-[\partial_{x}]_{T\mathcal{F}^{\pm}})=\mp d([s_{0}]_{\underline{\mathbb{C}}}-[\partial_{x}]_{\underline{\mathbb{C}}}),
        \]
		where $d\colon H^{1}(T)\rightarrow H^{2}(\mathbb{D}\times\mathbb{S}^{1})$ is the coboundary map for the pair $(\mathbb{D}\times\mathbb{S}^{1},T)$. $[s_{0}]_{\underline{\mathbb{C}}}-[\partial_{x}]_{\underline{\mathbb{C}}}\in H^{1}(T)$ is represented by the map $f\colon T\rightarrow\mathbb{S}^{1}$ given by $f(e^{\theta_{1}\sqrt{-1}},\theta_{2})=\theta_{1}$. Let $D$ be the meridian disk $\mathbb{D}\times 0$ with the complex orientation. We have
		\[
            \langle d([f]),[D]\rangle=\langle[f],[\partial D]\rangle=\langle[f],[\mu]\rangle=\deg(f|_{\mu})=1=\langle PD([\lambda]),[D]\rangle.
        \]
		Since $H_{1}(\mathbb{D}\times\mathbb{S}^{1})$ is torsion-free and $H_{2}(\mathbb{D}\times\mathbb{S}^{1},\partial\mathbb{D}\times\mathbb{S}^{1})\cong\mathbb{Z}$ is generated by $[D]$
		It follows that $d([f])=PD([\lambda])$. So we have
		\[
            e(T\mathcal{P}^{\pm},s_{0}|_{T})=\mp d([f])=\mp PD[\lambda].
        \]
		By Lemma \ref{Euler_class_for_standard_model}, we have
		\[
            e(T\mathcal{F}_{m,p,q}^{\pm},s_{0})=\pm PD(mq[\mu]+mp[\lambda])=\pm mp\cdot PD[\lambda].
        \]
		It follows from Lemma \ref{Gluing_formula}, Lemma \ref{Restriction_formula} that
		\[
            e(\mathcal{F}^{\pm},\partial)=\mp PD[\lambda]\pm (mp\cdot PD[\lambda])=\pm(mp-1)PD([\lambda]).
        \]
	\end{proof}
	
	\section{Proof of the main theorems}\label{non-geometric_case}
	
	In this section, we firstly prove Theorem \ref{graph_taut_foliation_distribution}, which determines the set of Euler classes of all taut foliations for graph manifolds whose JSJ pieces are products of $\mathbb{S}^{1}$ and positive genus compact surfaces. Then we use Theorem \ref{graph_taut_foliation_distribution} to prove the remaining theorems.
	
	\begin{Lem}\label{Contrain_on_product}
		Let $(M,\gamma)$ be a sutured manifold such that $M=\Sigma\times\mathbb{S}^{1}$ and $A(\gamma)=\emptyset$, where $\Sigma$ is a connected compact oriented surface.
		let $l$ be an oriented $\mathbb{S}^{1}$-fiber of $M$.
		Let $\mathcal{F}$ be a Reebless foliation on $(M,\gamma)$ compatible with the sutured structure.
		Then $e(\mathcal{F},\partial)=a\cdot PD[l]$ for some integer $a$ with $|a|\leq|\chi(\Sigma)|$ and $a\equiv\chi(\Sigma) \mod 2$.
	\end{Lem}
	
	\begin{proof}
		By K\"unneth formula, $H_{2}(M,\partial M)=H_{2}(\Sigma,\partial\Sigma)\oplus (H_{1}(\Sigma,\partial\Sigma)\otimes H_{1}(\mathbb{S}^{1}))$, Where $H_{2}(\Sigma,\partial\Sigma)\cong\mathbb{Z}$ is generated by $PD[l]$. Each class in $H_{1}(\Sigma)\otimes H_{1}(\mathbb{S}^{1})$ can be represented a properly embedded surface $S$ which is a union of essential tori and essential annuli. Applying Lemma \ref{Adjunction_inequality}, we have
		\[|\langle e(\mathcal{F},\partial),[S]\rangle|\leq-\chi(S)=0.\]
		So $e(\mathcal{F},\partial)=a\cdot PD[l]$, where  $a=\langle e(\mathcal{F},\partial),[\Sigma]\rangle\in\mathbb{Z}$.
		Applying Lemma \ref{Adjunction_inequality} again, we have
		$|a|\leq |\chi(\Sigma)|$ and $|a|\equiv\chi(\Sigma)\mod 2$.
	\end{proof}
	
	\begin{Thm}\label{graph_taut_foliation_distribution}
		Let $Y$ be a closed graph manifold such that each JSJ piece is a product of $\mathbb{S}^{1}$ and a positive genus compact surface. Let $N_{i}=\Sigma_{i}\times\mathbb{S}^{1}\ (1\leq i\leq n)$ be the JSJ pieces of $Y$. For each $i$, let $\chi_{i}=\chi(\Sigma_{i})$ and $l_{i}$ be an oriented $\mathbb{S}^{1}$-fiber of $N_{i}$. Then the set of Euler classes of all taut foliations on $Y$ is given by
		\[ \mathscr{F}(Y)=\left\{\sum\limits_{i=1}^{n}a_{i}\cdot PD[l_{i}]\colon|a_{i}|\leq -\chi_{i},\ a_{i}\equiv\chi_{i} \mod 2\right\}. \]
	\end{Thm}
	
	\begin{proof}
		Suppose that $\Sigma_{i}$ has genus $g_{i}$ with $k_{i}$ boundaries. Let $T_{u}\ (u=1,2,\ldots,m)$ be the JSJ tori of $Y$. Suppose that $T_{u}$ connects the JSJ pieces $N_{o(u)}$ and $N_{t(u)}$. Let $\lambda_{i}$ be a standard lamination on $\Sigma_{i}$. Then $\Lambda=\bigsqcup\limits_{i=1}^{n}\lambda_{i}\times\mathbb{S}^{1}$ is an oriented essential lamination on $Y$.
		
		Let $C_{ij}\ (1\leq i\leq n,\ 1\leq j\leq 2(g_{i}-1))$ and $C_{u}\ (1\leq u\leq m)$ be the components of the complementary region of $\Lambda$. Among these components, for each $1\leq i\leq n$, $C_{ij}\ (1\leq j\leq 2(g_{i}-1))$ are contained in $N_{i}$ and are products of $\mathbb{S}^{1}$ and $4$-gons; for each $1\leq u\leq m$, $C_{u}$ contains the JSJ torus $T_{u}$ and is given by gluing two products of $\mathbb{S}^{1}$ and one-holed $2$-gons along their common boundary $T_{u}$.
		
		Each $C_{ij}$ can be cut along $4$ vertical annuli into a sutured solid torus $(M_{ij},\gamma_{ij})$ and product regions $\mathcal{I}_{ij}\cong\mathbb{R}_{\geq 0}\times A(\gamma_{ij})$. $(M_{ij},\gamma_{ij})$ has $4$ longitudinal sutures. Let $l_{ij}$ be a core curve of $M_{ij}$. Orient $l_{ij}$ such that $[l_{ij}]=[l_{i}]\in H_{1}(N_{i})$. By Lemma \ref{Foliation_On_sutured_solid_torus}, for each $\epsilon_{ij}\in\{\pm1\}$, there is a taut foliation $\mathcal{F}_{ij}$ on $(M_{ij},\gamma_{ij})$ with
        \[
            e(\mathcal{F}_{ij},\partial)=\epsilon_{ij}[l_{ij}].
        \]
		
		Each $C_{u}$ can be cut along $4$ vertical annuli into a sutured torus shell $(M_{u},\gamma_{u})$ and product regions $\mathcal{I}_{u}\cong\mathbb{R}_{\geq 0}\times A(\gamma_{u})$. Let $l_{u}^{o}$ and $l_{u}^{t}$ be oriented simple closed curves in $T_{u}$ such that $[l_{u}^{o}]=[l_{o(u)}]\in H_{1}(N_{o(u)})$ and $[l_{u}^{t}]=[l_{t(u)}]\in H_{1}(N_{t(u)})$. Then $s(\gamma_{u})$ consists of two sutures of slope $[l_{u}^{o}]$ on one boundary and two sutures of slope $[l_{u}^{t}]$ on the other boundary. By Lemma \ref{Foliation_on_sutured_torus_shell}, for each $\epsilon_{u}^{o},\epsilon_{u}^{t}\in\{\pm1\}$, there is a taut foliation $\mathcal{F}_{u}$ on $(\mathcal{F}_{u},\gamma_{u})$ with
        \[
            e(\mathcal{F}_{u},\partial)=\epsilon_{u}^{o}PD[l_{u}^{o}]+\epsilon_{u}^{t}PD[l_{u}^{t}].
        \] 

        Denote the union of $(M_{ij},\gamma_{ij})$ and $(M_{u},\gamma_{u})$ by $(M,\gamma)$, the union of $\mathcal{F}_{ij}$ and $\mathcal{F}_{u}$ by $\mathcal{F}_{M}$, and the union of $\mathcal{I}_{ij}$ and $\mathcal{I}_{u}$ by $\mathcal{I}$. Then $\mathcal{F}_{M}$ meets $A(\gamma)$ in a $1$-dimensional foliation $f$. Any identification $\mathcal{I}\cong\mathbb{R}_{\geq 0}\times A(\gamma)$ produces a $2$-dimensional foliation $\mathcal{F}_{\mathcal{I}}=\mathbb{R}_{\geq0}\times f$ on $\mathcal{I}$. Then $\mathcal{F}_{M},\mathcal{F}_{\mathcal{I}}$ and $\Lambda$ together form an oriented foliation $\mathcal{F}$ on $Y$. $\mathcal{F}$ is taut since it has no torus leaves.
        
        \begin{Claim}
            By selecting an appropriate identification, we can ensure that $\mathcal{F}$ is $C^{\infty,0}$.
        \end{Claim}

        This fact is intuitively obvious but is not readily justified in a few words. We include a proof here for the sake of completeness.

        \begin{proof}[Proof of Claim]
            It suffices to ensure that $T\mathcal{F}_{\mathcal{I}}$ can be extended to a $C^{0}$-plane field over the closure of $\mathcal{I}$ in $Y$.

            Denote by $\theta\in\mathbb{R}/2\pi\mathbb{Z}$ the coordinate of $\mathbb{S}^{1}$. Choose an arbitrary Riemannian metric on each $\Sigma_{i}$. These metrics induce metrics on the projectivized tangent bundles $PT\Sigma_{i}=\cup_{p\in\Sigma_{i}}\mathbb{P}(T_{p}\Sigma_{i})$, a product metric on $Y\backslash M$, and further a metric on the Grassmann bundle $\mathrm{Gr}_{2}(T(Y\backslash M))=\cup_{p\in Y\backslash M}\mathrm{Gr}_{2}(T_{p}Y)$.
    
            For each component $\mathcal{I}_{c}$ of $\mathcal{I}$, $\mathcal{I}_{c}$ takes the form $c\times\mathbb{S}^{1}$, where $c$ is a geometric cusp of some complementary region of $\lambda_{i}\subset\Sigma_{i}$. Then $A(\gamma_{c})=(c\times\mathbb{S}^{1})\cap M$ is a component of $A(\gamma)$, and $\mathcal{F}_{M}$ meets $A(\gamma_{c})$ in a $1$-dimensional foliation $f_{c}$.
            Pick an identification $c\cong\mathbb{R}_{\geq0}\times [0,1]=\{(u,v)\colon u\in\mathbb{R}_{\geq0},v\in[0,1]\}$ such that
            \begin{enumerate}
                \item $A(\gamma_{c})=0\times[0,1]\times\mathbb{S}^{1}$.
                \item $\lim\limits_{a\rightarrow+\infty}\max\{\|\partial_{v}|_{(a,b)}\|\colon b\in[0,1]\}=0$.
                \item $\lim\limits_{a\rightarrow+\infty}\mathrm{diam}_{PT\Sigma_{i}}(\{[\mathbb{R}\cdot\partial_{u}|_{(a,b)}]\colon b\in [0,1]\})=0$.
            \end{enumerate} 
            Then we obtain an identification $\mathcal{I}_{c}\cong \mathbb{R}_{\geq0}\times[0,1]\times\mathbb{S}^{1}\cong\mathbb{R}_{\geq 0}\times A(\gamma_{c})$. Let $\mathcal{F}_{c}=\mathbb{R}_{\geq 0}\times f_{c}$ be the corresponding $2$-dimensional foliation on $\mathcal{I}_{c}$.
    
            Express $Tf_{c}$ as $Tf_{c}|_{(0,b,h)}=\mathrm{Span}_{\mathbb{R}}\{(\partial_{\theta}+\psi(b,h)\partial_{v})|_{(0,b,h)}\}$, where $\psi\in C^{\infty}([0,1]\times\mathbb{S}^{1})$. Then we have $T\mathcal{F}_{c}|_{(a,b,h)}=\mathrm{Span}_{\mathbb{R}}\{\partial_{u}|_{(a,b,h)},(\partial_{\theta}+\psi(b,h)\partial_{v})|_{(a,b,h)}\}$. Let $V$ be the plane field over $\mathcal{I}_{c}$ spanned by the vector fields $\partial_{v}$ and $\partial_{\theta}$. Then
            \[
                 d_{\mathrm{Gr}_{2}(T(Y\backslash M))}([T\mathcal{F}_{c}|_{(a,b,h)}],[V|_{(a,b,h)}]) \leq |\psi(b,h)|\ \|\partial_{v}|_{(a,b)}\|.
            \]
            It follows from condition (2) that the distance between $[T\mathcal{F}_{c}|_{p}]$ and $[V_{p}]$ tends to zero as the $u$-coordinate of $p$ tends to infinity. Moreover, condition (3) guarantees that $V$ can be extended to a $C^{0}$-plane field over the closure of $\mathcal{I}_{c}$ in $Y$. We conclude that $T\mathcal{F}_{c}$ can be extended to a $C^{0}$-plane field over the closure of $\mathcal{I}_{c}$ in $Y$.
        \end{proof}
        
	    It turns out that $\mathcal{F}$ is an oriented $C^{\infty,0}$-taut foliation on $Y$. Next, calculate the Euler class of $\mathcal{F}$. Let $\mathcal{P}_{i}$ be the product foliation on $N_{i}$ foliated by $\Sigma_{i}\times p\ (p\in\mathbb{S}^{1})$.
	    Our choice of identification $\mathcal{I}\cong\mathbb{R}_{\geq 0}\times A(\gamma)$ can ensure that $T\mathcal{F}$ is transverse to  $T\mathcal{P}_{i}$ on $N_{i}$ outside $(M,\gamma)$. Then the intersection of $T\mathcal{F}$ with $T\mathcal{P}_{i}\ (1\leq i\leq n)$ yields a unit section $s$ of $T\mathcal{F}|_{Y-\mathrm{int}(M)}$.
	    It is straightforward to check that $s|_{\partial M}$ is a canonical trivialization of $T\mathcal{F}|_{\partial M}$. It follows from Lemma \ref{Gluing_formula} and Lemma \ref{Restriction_formula} that
		
		\begin{align*}
			e(\mathcal{F})=&\sum\limits_{i=1}^{n}\sum\limits_{j=1}^{2(g_{i}-1)}e(\mathcal{F}_{ij},\partial)+\sum\limits_{u=1}^{m}e(\mathcal{F}_{u},\partial) \\
			=&\sum\limits_{i=1}^{n}\sum\limits_{j=1}^{2(g_{i}-1)}\epsilon_{ij}PD[l_{ij}]+\sum\limits_{u=1}^{m}(\epsilon_{u}^{o}PD[l_{u}^{o}]+\epsilon_{u}^{t}PD[l_{u}^{t}])\\
			=&\sum\limits_{i=1}^{n}\left(\sum\limits_{j=1}^{2(g_{i}-1)}\epsilon_{ij}+\sum\limits_{o(u)=i}\epsilon_{u}^{o}+\sum\limits_{t(u)=i}\epsilon_{u}^{t}\right)PD[l_{i}].
		\end{align*}
        
		Since $ \sum\limits_{j=1}^{2(g_{i}-1)}1+\sum\limits_{o(u)=i}1+\sum\limits_{t(u)=i}1=2(g_{i}-1)+k_{i}=-\chi_{i} $, and $\epsilon_{ij},\epsilon_{u}^{o},\epsilon_{u}^{t}\in\{\pm1\}$ can be taken arbitrarily, every class in $\left\{\sum\limits_{i=1}^{n}a_{i}\cdot PD[l_{i}]\colon a_{i}\equiv\chi_{i}\ \mathrm{mod}\ 2,\ |a_{i}|\leq -\chi_{i}\right\}$ can be realized.
		
		Conversely, let $\mathcal{F}$ be a taut foliation on $Y$.
		By the main theorem of \cite{BR99}, possibly after blowing up along finite number of leaves, which does not change the Euler class or effect tautness, $\mathcal{F}$ can be isotoped such that each JSJ torus is either a leaf of $\mathcal{F}$ or is transverse to $\mathcal{F}$.
        
        Denote by $\mathcal{F}_{i}$ the restriction of $\mathcal{F}$ to $N_{i}$. Then $\mathcal{F}_{i}$ is Reebless. Besides, $\mathcal{F}_{i}$ specifies a sutured structure $\delta_{i}$ on $N_{i}$ as follows: $A(\delta_{i})=\emptyset$, $T(\delta_{i})=\partial_{\pitchfork}Y$, $R_{+}(\delta_{i})$ (resp. $R_{-}(\delta_{i})$) are union of boundary leaves of $\mathcal{F}_{i}$ whose normal vectors (with respect to the leaf orientations) point out of (resp. into) $N_{i}$.
        
		Suppose that $\mathcal{F}|_{T_{u}}$ contains Reeb annuli for some $u$. Then the slope of $\mathcal{F}|_{T_{u}}$ is vertical with respect to both $N_{o(u)}$ and $N_{t(u)}$ by \cite[Proposition 2]{BNR97} (see also \cite{Bri99}), contradicting to the minimality of JSJ decomposition.
		So $\mathcal{F}|_{T_{u}}$ has no Reeb annuli.
        Thus, $\mathcal{F}_{i}$ is compatible with the sutured structure of $(N_{i},\delta_{i})$. 
		By Lemma \ref{Contrain_on_product}, $e(\mathcal{F}_{i},\partial)=a_{i}\cdot PD[l_{i}]$ for some integer $a_{i}$ with $|a_{i}|\leq |\chi(\Sigma_{i})|=-\chi_{i}$ and $a_{i}\equiv \chi_{i}\mod 2$. 
        It follows that
        \[
            e(\mathcal{F})=\sum\limits_{i=1}^{n}e(\mathcal{F}_{i},\partial)=\sum\limits_{i=1}^{n}a_{i}\cdot PD[l_{i}].
        \]
		This completes the proof.
	\end{proof}
	
	Before the proof of Theorem \ref{Main_Thm}, we need one more lemma, which implies that any rational point in the dual Thurston norm unit ball can be reduced to the desired form after passing through some finite cover.
	
	\begin{Lem}\label{virtual_class_realize}
		Let $Y$ be a non-geometric closed graph manifold. Let $w$ be a rational point in $\mathcal{B}_{\mathrm{Th}^{*}}(Y)$. Then there is a finite cover $\pi\colon\tilde{Y}\rightarrow Y$ such that
		\begin{itemize}
			\item[(1)]
			Let $N_{i}\ (1\leq i\leq n)$ be the JSJ pieces of $\tilde{Y}$. Then each $N_{i}$ is a product of $\mathbb{S}^{1}$ and a positive genus surface $\Sigma_{i}$.
			\item[(2)]
			Let $\chi_{i}=\chi(\Sigma_{i})$ and $l_{i}$ be an oriented $\mathbb{S}^{1}$-fiber of $N_{i}$. There exist integers $a_{1},\ldots,a_{n}$ such that $a_{i}\equiv\chi_{i}\ \mathrm{mod}\ 2,\ |a_{i}|\leq-\chi_{i}$ and $\pi^{*}w=\sum\limits_{i=1}^{n}a_{i}\cdot PD[l_{i}]\in H^{2}(\tilde{Y};\mathbb{Q})$.
		\end{itemize}
	\end{Lem}
	
	\begin{proof}
		By Lemma \ref{first_virtual_reduction}, there exists a finite cover $\pi'\colon Y'\rightarrow Y$ such that each JSJ piece of $Y'$ is a product of $\mathbb{S}^{1}$ and a positive genus surface.
		Let $N_{i}'=\Sigma_{i}'\times\mathbb{S}^{1}\ (1\leq i\leq n')$ be the JSJ pieces of $Y'$. Let $\chi_{i}'=\chi(\Sigma_{i}')$ and $l_{i}'$ be an oriented regular fiber of $N_{i}'$.
		By Lemma \ref{Char_rational_point}, there exist $a_{i}'\in\mathbb{Q}, |a_{i}'|\leq-\chi_{i}$ such that
		\[
            (\pi')^{*}w'=\sum\limits_{i=1}^{n'}a_{i}'PD[l_{i}']\in H^{2}(Y';\mathbb{Q}).
        \]
		By Lemma \ref{characteristic_cover}, there is an integer $m_{0}$ such that for any integral multiple $m$ of $m_{0}$, there is a JSJ $m$-characteristic cover of $Y'$. Choose an integer $m$ such that $m_{0}|m$ and $ma_{i}'$ are integers for each $i$. Let $\tilde{\pi}\colon\tilde{Y}\rightarrow Y'$ be a JSJ $2m$-characteristic cover.
		Let $\tilde{N}_{j}=\tilde{\Sigma}_{j}\times\mathbb{S}^{1}\ (1\leq j\leq \tilde{n})$ be the JSJ pieces of $\tilde{Y}$. Let $\tilde{\chi}_{j}=\chi(\tilde{\Sigma}_{j})$ and $\tilde{l}_{j}$ be an oriented regular fiber of $\tilde{N}_{j}$. Each $\tilde{N}_{j}$ covers some JSJ piece $N_{i}'$. Then each boundary component of $\tilde{\Sigma}_{j}$ covers some boundary component of $\Sigma_{i}'$ $2m$ times. So the covering degree $d_{j}$ of $\tilde{\Sigma}_{j}\rightarrow\Sigma_{i}'$ is divided by $2m$. Let $\tilde{a}_{j}=d_{j}a_{i}'$. Then $|\tilde{a}_{j}|=d_{j}|a_{i}'|\leq-d_{j}\chi_{i}'=-\tilde{\chi}_{j}$. Besides, we have $2ma_{i}'|\tilde{a}_{j}$ and $2m|\tilde{\chi}_{j}$. So $\tilde{a}_{j}$ is an integer and $\tilde{a}_{j}\equiv\tilde{\chi}_{j}\equiv 0\ \mathrm{mod}\ 2$.
		Note that the preimage of $a_{i}'$ fibers of $N_{i}'$ in $\tilde{N}_{j}$ consists of $d_{j}a_{i}'=\tilde{a}_{j}$ fibers of $\tilde{N}_{j}$. So
        \[
            (\pi'\circ\tilde{\pi})^{*}w=\tilde{\pi}^{*}\left(\sum\limits_{i=1}^{n'}a_{i}'PD[l_{i}']\right)=\sum\limits_{j=1}^{\tilde{n}}\tilde{a}_{j}PD[\tilde{l}_{j}]\in H^{2}(\tilde{Y};\mathbb{Q}).
        \]
		Therefore, $\pi=\pi'\circ\tilde{\pi}\colon\tilde{Y}\rightarrow Y$ is the desired cover.
	\end{proof}
	
	\begin{Rem}
		By \cite[Lemma 2.9]{LeeLip08}, if $w$ is a integral class, then after passing to a further cover, we can achieve that $\pi^{*}w=\sum\limits_{i=1}^{n}a_{i}\cdot PD[l_{i}]\in H^{2}(\tilde{Y};\mathbb{Z})$.
	\end{Rem}
	
	\begin{proof}[Proof of Theorem \ref{Main_Thm}]
		If $Y$ is Seifert-fibered, then there is finite cover $\pi'\colon Y'\rightarrow Y$ such that $Y'$ is a $\mathbb{S}^{1}$-bundle over a closed surface $\Sigma$.
		Let $e(Y')$ denote the Euler number. The aspherical condition implies that $Y'$ does not support $\mathbb{S}^{3}$-geometry and $\mathbb{S}^{2}\times\mathbb{R}$-geometry, so $g(\Sigma)\geq 1$. In particular, $b_{1}(Y')\geq 1$. Set $w'=(\pi')^{*}w$.
		If $e(Y')=0$, then $Y'\cong\Sigma\times\mathbb{S}^{1}$. So $w'=a\cdot PD[\mathbb{S}^{1}]$ for some $a\in\mathbb{Q}$ with $|a|\leq -\chi(\Sigma)$ by Lemma \ref{Char_rational_point}. Choose an positive integer $m$ such that $\tilde{a}=ma\in 2\mathbb{Z}$. Choose a finite cover $\tilde{\pi}\colon\tilde{Y}=\tilde{\Sigma}\times\mathbb{S}^{1}\rightarrow Y=\Sigma\times\mathbb{S}^{1}$ whose horizontal covering degree is $m$. Then $\tilde{\pi}^{*}w'=\tilde{a}\cdot PD[\mathbb{S}^{1}]$. We have $|\tilde{a}|\leq -m\chi(\Sigma)=-\chi(\tilde{\Sigma})$ and $\tilde{a}\equiv 0\mod 2$. By Theorem \ref{graph_taut_foliation_distribution}, there exists a taut foliation $\mathcal{F}$ on $\tilde{Y}$ such that $e(\mathcal{F})=\tilde{\pi}^{*}w'$.
        If $e(Y')\not=0$, then $Y'$ has vanishing Thurston norm. So $w'\in\mathcal{B}_{\mathrm{Th}^{*}}(Y')=\{0\}$. Since $b_{1}(Y')\geq 1$, there exists a taut foliation $\mathcal{F}$ on $Y'$ \cite[Theorem 5.5]{Gab83}. Then we have $e(\mathcal{F})=0=w'$.
		
		If $Y$ supports $\mathrm{Sol}$-geometry, then there is a finite cover $\pi\colon\tilde{Y}\rightarrow Y$ such that $\tilde{Y}$ is a torus bundle over $\mathbb{S}^{1}$ of Anosov type.
		Then $H_{2}(\tilde{Y})\cong\mathbb{Z}$ is generated by the fibered class, whose Thurston norm is zero. So $\tilde{Y}$ has vanishing Thurston norm. Thus $\pi^{*}w\in\mathcal{B}_{\mathrm{Th}^{*}}(\tilde{Y})=\{0\}$.
	    Let $\mathcal{F}$ be the taut foliation on $\tilde{Y}$ foliated by torus fibers. Then we have $e(\mathcal{F})=0=\pi^{*}w$.
		
		If $Y$ is not geometric, then pick a finite cover $\pi\colon\tilde{Y}\rightarrow Y$ which satisfies the conditions in Lemma \ref{virtual_class_realize}. By Theorem \ref{graph_taut_foliation_distribution}, $\pi^{*}w$ is the Euler class of some taut foliation on $\tilde{Y}$.
	\end{proof}
	
	\begin{proof}[Proof of Theorem \ref{counterexample_real_Euler_class_zero}]
		Fix an integer $g\geq 1$. Choose a integer $d>2g+1$ and three integers $d_{1},d_{2},d_{3}$ such that $\gcd(d,d_{1})=\gcd(d,d_{2})=\gcd(d,d_{3})=1$ and $d_{1}+d_{2}+d_{3}=0$. Let $\Sigma'$ and $\Sigma''$ be two genus $g$ compact oriented surfaces with three boundary components. We have
        \[
            \chi(\Sigma')=\chi(\Sigma'')=-1-2g.
        \]
        Write $\partial\Sigma'=m_{1}'\cup m_{2}'\cup m_{3}'$ and $\partial\Sigma''=m_{1}''\cup m_{2}''\cup m_{3}''$. Equip $m_{i}',m_{i}''$ with the boundary orientations. Let
        \[
            N'=\Sigma'\times\mathbb{S}^{1},\ N''=\Sigma''\times\mathbb{S}^{1},\ T_{i}'=m_{i}'\times\mathbb{S}^{1},\ T_{i}''=m_{i}''\times\mathbb{S}^{1}.
        \]
        Let $l_{i}'$ (resp. $l_{i}''$) be an $\mathbb{S}^{1}$-fiber in $T_{i}'$ (resp. $T_{i}''$). Orient $l_{i}',l_{i}''$ such that
        \[
            PD[m_{i}']([l_{i}'])=PD[m_{i}'']([l_{i}''])=1.
        \]
        Pick an orientation-reversing homeomorphism $\psi_{i}\colon T_{i}'\rightarrow T_{i}''$ such that
        \[
            (\psi_{i})_{*}[m_{i}']=d[m_{i}'']+d_{i}[l_{i}'']\ (i=1,2,3).
        \]
        Then $\psi=\psi_{1}\sqcup\psi_{2}\sqcup\psi_{3}$ is an orientation-reversing homeomorphism $\psi\colon\partial N'\rightarrow\partial N''$. Let $Y=N'\cup_{\psi}N''$. $T_{i}'$ and $T_{i}''$ glue to a torus $T_{i}$ in $Y$.
        
        From the construction, $Y$ is a closed graph manifold with three JSJ tori $T_{1},T_{2},T_{3}$ and two JSJ pieces $N',N''$.
		Let
        \[
            i\colon N'\cap N''\rightarrow Y,\ i'\colon N'\cap N''\rightarrow N',\ i''\colon N'\cap N''\rightarrow N''
        \]
        denote the inclusions.
		The Mayer-Vietoris sequence for $(Y,N',N'')$ gives an exact sequence
		\[\cdots\rightarrow H_{2}(Y)\xrightarrow{\phi} H_{1}(N'\cap N'')\xrightarrow{(i'_{*},i''_{*})} H_{1}(N')\oplus H_{1}(N'')\rightarrow H_{1}(Y)\rightarrow\cdots.\]
		In particular, we have
		\[b_{1}(Y)\geq b_{1}(N')+b_{1}(N'')-b_{1}(N'\cap N'')=(2g+3)+(2g+3)-6=4g.\]
		Set $\alpha=[m_{1}']+[m_{2}']+[m_{3}']\in H_{1}(N'
		\cap N'')$.
		Then
		\begin{align*}
			i'_{*}(\alpha)&=[m_{1}']+[m_{2}']+[m_{3}']=0\in H_{1}(N'),\\
			i''_{*}(\alpha)&=(d[m_{1}'']+d_{1}[l_{1}''])+(d[m_{2}'']+d_{2}[l_{2}''])+(d[m_{3}'']+d_{3}[l_{3}''])\\
			&=d([m_{1}'']+[m_{2}'']+[m_{3}''])+(d_{1}+d_{2}+d_{3})[l_{1}'']\\
			&=d([m_{1}'']+[m_{2}'']+[m_{3}''])=0\in H_{1}(N'').
		\end{align*}
		Thus, there exists $\beta\in H_{2}(Y)$ such that $\phi(\beta)=\alpha$.
		Note that $\phi$ is the composition
        \[
            H_{2}(Y)\xrightarrow{PD_{Y}}H^{1}(Y)\xrightarrow{i^{*}}H^{1}(N'\cap N'')\xrightarrow{PD_{N'\cap N''}}H_{1}(N'\cap N'').
        \]
		So for any $u\in H_{1}(N'\cap N'')$, we have
		\begin{align*}
            PD_{Y}(i_{*}u)(\beta)&=PD_{Y}(\beta)(i_{*}u)=i^{*}(PD_{Y}(\beta))(u)=PD_{N'\cap N''}(\phi(\beta))(u)\\&=PD_{N'\cap N''}(\alpha)(u).
        \end{align*}
		In particular, we have
        \begin{align*}
            PD_{Y}[l_{i}'](\beta)&=PD_{N'\cap N''}(\alpha)(l_{1}')=PD_{T_{1}'}([m_{1}'],[l_{1}'])=1,\\
            PD_{Y}[l_{1}''](\beta)&=PD_{N'\cap N''}(\alpha)([l_{1}''])=PD_{T_{1}''}((\psi_{1})_{*}[m_{1}'],[l_{1}''])=d.
        \end{align*}
        For any taut foliation $\mathcal{F}$ on $Y$, by Theorem \ref{graph_taut_foliation_distribution}, there exists integers $a',a''\in\{\pm1,\pm3,\ldots,\pm(2g+1)\}$ such that
        \[
            e(\mathcal{F})=a'PD[l_{1}']+a''PD[l_{1}''].
        \]
        Then
        $e(\mathcal{F})(\beta)=a'+da''\not=0$ since $d>2g+1$. So $e(\mathcal{F})\not=0\in H^{2}(Y,\mathbb{R})$.
	\end{proof}
	
	\begin{proof}[Proof of Theorem \ref{counterexample_integral_Euler_class_zero}]
		Fix an integer $g\geq 1$. Let $\Sigma'$ and $\Sigma''$ be two genus $g$ compact oriented surfaces with a single boundary component. We have
        \[
            \chi(\Sigma')=\chi(\Sigma'')=1-2g.
        \]
        Let $m'=\partial\Sigma'$ and $m''=\partial\Sigma''$. Equip $m',m''$ with the boundary orientations. Let
        \[
            N'=\Sigma'\times\mathbb{S}^{1},\ N''=\Sigma''\times\mathbb{S}^{1}.
        \]
        Let $l'$ (resp. $l''$) be an $\mathbb{S}^{1}$-fiber in $\partial N'$ (resp. $\partial N''$). Orient $l',l''$ such that
        \[
            PD[m'][l']=PD[m''][l'']=1.
        \]
        Pick a orientation-reversing homeomorphism $\psi\colon\partial N'\rightarrow\partial N''$ such that
        \[
            \psi_{*}[m']=2g[m'']+(4g^{2}+1)[l''],\ \psi_{*}[l']=[m'']+2g[l''].
        \]
        Let $Y=N'\cup_{\psi}N''$. $\partial N'$ and $\partial N''$ glue to a torus $T$ in $Y$. $Y$ is a closed manifold with a single JSJ torus $T$ and two JSJ pieces $N',N''$. Let
        \[
            i\colon T\rightarrow Y,\ i'\colon T\rightarrow N',\ i''\colon T\rightarrow N''
        \]
        denote the inclusions. We have $\mathrm{Ker}\ i'_{*}=\mathbb{Z}\cdot[m_{1}']$ and
		$\mathrm{Ker}\ i''_{*}=\mathbb{Z}\cdot[m_{1}'']$. The (reduced) Mayer-Vietoris sequence for $(Y,N',N'')$ gives an exact sequence
		\[
            \cdots\rightarrow H_{1}(T)\xrightarrow{(i'_{*},-i''_{*})} H_{1}(N')\oplus H_{1}(N'')\rightarrow H_{1}(Y)\rightarrow 0.
        \]
        It follows that
        \[
            b_{1}(Y)=b_{1}(N')+b_{1}(N'')-b_{1}(T)=(2g+1)+(2g+1)-2=4g.
        \]
		and
		\[\mathrm{Ker}\ i_{*}=\mathrm{Ker}\ i'_{*}+\mathrm{Ker}\ i''_{*}=\mathbb{Z}\cdot[m']\oplus\mathbb{Z}\cdot[m'']. \]
		Consider the homomorphism $\Phi\colon H_{1}(T)\rightarrow\mathbb{Z}/(4g^{2}+1)\mathbb{Z}$ given by 
        \[
          \Phi(a[m'']+b[l''])=[b].
        \]
		Then we have
        \[
            \Phi([m'])=[0],\ \Phi([l'])=[2g],\ \Phi([m''])=[0],\ \Phi([l''])=[1].
        \]
		So $\Phi$ induces an isomorphism
        \[\Phi\colon\mathrm{Im}\ i_{*}\cong H_{1}(T)/\mathrm{Ker}\ i_{*}\xrightarrow{\cong}\mathbb{Z}/(4g^{2}+1)\mathbb{Z}.\]
		For any taut foliation $\mathcal{F}$ on $Y$, by Theorem \ref{graph_taut_foliation_distribution}, there exists integers $a',a''\in\{\pm1,\pm3,\ldots,\pm(2g-1)\}$ such that
        \[
            e(\mathcal{F})=a'PD[l']+a''PD[l''].
        \]
        So $PD(e(\mathcal{F}))$ lies in $ \mathrm{Im}\ i_{*}$, which is torsion. Thus, $e(\mathcal{F})=0\in H^{2}(Y;\mathbb{R})$. On the other hand, $\Phi(PD(e(\mathcal{F})))=[a'+2ga'']\not=[0]$. So $e(\mathcal{F})\not=0\in H^{2}(Y;\mathbb{Z})$.
	\end{proof}

	\bibliographystyle{alpha}
	\bibliography{reference}
	
\end{document}